\documentclass[12pt]{amsart}
\usepackage{SSdefn}
\usepackage{tikz}
\usepackage{comment}
\usetikzlibrary{positioning,shapes,arrows,mindmap,trees}

\newcommand{\DOI}[1]{\href{http://doi.org/#1}{\color{purple}{\tiny\tt DOI:#1}}}
\newcommand{\arxiv}[1]{\href{http://arxiv.org/abs/#1}{{\tiny\tt arXiv:#1}}}

\tikzstyle{arrow} = [-,>=stealth]
\tikzset{node/.style={%
      draw,
      circle,
      fill,
      inner sep=0,
      outer sep=0,
      minimum size=2pt,
      node distance=30pt,
}}

\makeatletter
\def\thickhline{%
  \noalign{\ifnum0=`}\fi\hrule \@height \thickarrayrulewidth \futurelet
   \reserved@a\@xthickhline}
\def\@xthickhline{\ifx\reserved@a\thickhline
               \vskip\doublerulesep
               \vskip-\thickarrayrulewidth
             \fi
      \ifnum0=`{\fi}}
\makeatother

\newlength{\thickarrayrulewidth}
\newcommand{\defi}[1]{\textbf{#1}}

\DeclareMathOperator{\Div}{Div}
\DeclareMathOperator{\avg}{avg}

\newcommand{\cref}[2]{%
  (\hyperref[#1]{\protect\NoHyper\ref{#1}\protect\endNoHyper#2})%
}

\newcommand{\FB}{\mathbf{FB}}

\newcommand{\FI}{\mathbf{FI}}
\newcommand{\FS}{\mathbf{FS}}
\newcommand{\FA}{\mathbf{FA}}
\newcommand{\VI}{\mathbf{VI}}
\newcommand{\VIC}{\mathbf{VIC}}
\newcommand{\VA}{\mathbf{VA}}
\newcommand{\OB}{\mathbf{OB}}

\newcommand{\OS}{\mathbf{OS}}
\newcommand{\VB}{\mathbf{VB}}

\newcommand{\bone}{\mathbf{1}}

\title[Brauer categories II: curried algebra]{The representation theory of Brauer categories II: curried algebra}
\date{July 10, 2022}

\author{Steven V Sam}
\address{Department of Mathematics, University of California, San Diego, CA}
\email{\href{mailto:ssam@ucsd.edu}{ssam@ucsd.edu}}
\urladdr{\url{http://math.ucsd.edu/~ssam/}}
\thanks{SS was supported by NSF grant DMS-1849173.}

\author{Andrew Snowden}
\address{Department of Mathematics, University of Michigan, Ann Arbor, MI}
\email{\href{mailto:asnowden@umich.edu}{asnowden@umich.edu}}
\urladdr{\url{http://www-personal.umich.edu/~asnowden/}}
\thanks{AS was supported by NSF grant DMS-1453893.}

\begin{document}

\begin{abstract}
A representation of $\fgl(V)=V \otimes V^*$ is a linear map $\mu \colon \fgl(V) \otimes M \to M$ satisfying a certain identity. By currying, giving a linear map $\mu$ is equivalent to giving a linear map $a \colon V \otimes M \to V \otimes M$, and one can translate the condition for $\mu$ to be a representation to a condition on $a$. This alternate formulation does not use the dual of $V$, and makes sense for any object $V$ in a tensor category $\cC$. We call such objects representations of the \emph{curried general linear algebra} on $V$. The currying process can be carried out for many algebras built out of a vector space and its dual, and we examine several cases in detail. We show that many well-known combinatorial categories are equivalent to the curried forms of familiar Lie algebras in the tensor category of linear species; for example, the titular Brauer category ``is'' the curried form of the symplectic Lie algebra. This perspective puts these categories in a new light, has some technical applications, and suggests new directions to explore.
\end{abstract}

\maketitle
\tableofcontents

\section{Introduction}

\subsection{Curried algebra}

Let $V$ be a finite-dimensional vector space. The general linear Lie algebra on $V$, denoted $\fgl(V)$, can be identified with the tensor product $V \otimes V^*$. A representation of $\fgl(V)$ on a vector space $M$ is a linear map
\begin{displaymath}
\mu \colon \fgl(V) \to \End(M)
\end{displaymath}
satisfying the equation
\begin{equation} \label{eq:gl1}
\mu([X,Y]) = [\mu(X), \mu(Y)]
\end{equation}
for all $X,Y \in \fgl(V)$. By currying (also known as tensor-hom adjunction), giving the linear map $\mu$ is equivalent to giving a linear map
\begin{displaymath}
a \colon V \otimes M \to V \otimes M.
\end{displaymath}
A natural problem, then, is to determine what condition \eqref{eq:gl1} corresponds to in terms of $a$. In Proposition~\ref{prop:glid}, we find that it amounts to the identity
\begin{equation} \label{eq:gl2}
\tau a \tau a - a \tau a \tau = a \tau - \tau a
\end{equation}
in $\End(V \otimes V \otimes M)$, where $\tau$ is the map that switches the first two tensor factors and we have written $a$ for $\id \otimes a$. We thus have two equivalent ways of viewing representations of $\fgl(V)$: as linear maps $\mu$ satisfying \eqref{eq:gl1}, or as a linear map $a$ satisfying \eqref{eq:gl2}.

The advantage of the second point of view is that it makes sense in contexts where we may not have duals. Indeed, suppose that $V$ is an object in a tensor category $\cC$. We define the \defi{curried general linear Lie algebra} on $V$, denoted $\ul{\fgl}(V)$, in a Tannakian sense: a $\ul{\fgl}(V)$-module is a map $a \colon V \otimes M \to V \otimes M$ satisfying \eqref{eq:gl2}. We emphasize that $\ul{\fgl}(V)$ is not actually an object of $\cC$: only the notion of $\ul{\fgl}(V)$-module is defined. If $V$ is dualizable then one can form the Lie algebra $\fgl(V)=V \otimes V^*$ in $\cC$, and $\ul{\fgl}(V)$-modules are equivalent to $\fgl(V)$-modules. However, one can consider $\ul{\fgl}(V)$-modules even if $V$ is not dualizable.

The above process can be applied to many algebras built out of a vector space and its dual, and we examine a number of cases in detail. Our primary motivation for developing this theory lies with its applications to representations of combinatorial categories, which we now explain.

\subsection{Representations of combinatorial categories}

Let $\fG$ a category and let $\bk$ be a commutative ring. A \defi{$\fG$-module} is a functor $\fG \to \Mod_\bk$.  Representations of categories, especially those of a combinatorial flavor, have received extensive attention in the last decade, and (for certain $\fG$'s) form the main subject of this series of papers. We now describe how curried algebras can be used to better understand these objects.

Let $\FB$ be the category of finite sets and bijections. An $\FB$-module, also known as a linear species, is simply a sequence of symmetric group representations. Given two $\FB$-modules $M$ and $N$, we define their tensor product $V \otimes W$ to be the $\FB$-module given by
\[
  (M \otimes N)(S) = \bigoplus_{S=A \amalg B} M(A) \otimes N(B).
\]
This gives the category of $\FB$-modules a symmetric monoidal structure. The motivating problem for this paper is the following: given a diagram category $\fG$, express $\fG$-modules as $\FB$-modules with extra structure, defined in terms of the tensor product. The curried perspective will help us understand this extra structure.

Here is the simplest case (which does not require currying). Following Church, Ellenberg, and Farb \cite{fimodule}, let $\FI$ be the category of finite sets and injections. An $\FI$-module is a sequence of symmetric group representations (i.e., an $\FB$-module) with some transition maps. Let $\bV$ be the \defi{standard} $\FB$-module: this is $\bk$ on sets of size~1 and~0 on all other sets. It turns out that the transition maps in $M$ can be encoded as a map of $\FB$-modules $a \colon \bV \otimes M \to M$. Not every such map $a$ defines an $\FI$-module: the key condition is that $a$ should give $M$ the structure of a $\Sym(\bV)$-module. This perspective led to a rich analysis of the category of $\FI$-modules in \cite{symc1}.

We now look at a slightly more complicated case. Consider the category $\FI\sharp$, also introduced by Church, Ellenberg, and Farb. Its objects again are finite sets, but now a morphism $S \to T$ is a pair $(S_0, i)$ where $S_0$ is a subset of $S$ and $i \colon S_0 \to T$ is an injection. An $\FI\sharp$-module $M$ is an $\FB$-module equipped with transition maps $M([n]) \to M([n+1])$, corresponding to the standard inclusion $[n] \to [n+1]$, and $M([n+1]) \to M([n])$, corresponding to the standard partial injection $[n+1] \to [n]$ defined on $[n]$. These transition maps can be encoded as maps of $\FB$-modules
\begin{displaymath}
a \colon \bV \otimes M \to M, \qquad b \colon M \to \bV \otimes M
\end{displaymath}
Not every pair $(a,b)$ defines an $\FI\sharp$-module structure on $M$: there are a few conditions that must be satisfied (see \cite{curried}). Initially, these conditions do not seem to have much meaning. This is where the curried perspective comes in: it turns out that the conditions for $(a,b)$ to define an $\FI\sharp$-module are (nearly) the conditions needed for it to define a representation of the curried Weyl algebra.

This phenomenon occurs throughout this paper. In particular, we find that representations of all of the Brauer-like categories of interest in this series of papers can be viewed as representations of the curried forms of familiar Lie algebras. For example, representations of the Brauer category itself are equivalent to representations of the curried symplectic Lie algebra $\ul{\fsp}$ in $\Mod_{\FB}$. See Figure~\ref{fig1} for a summary. The details for the examples not appearing in the current article can be found in \cite{curried}.

\begin{figure}[!h]
\begin{tabular}{ll}
\thickhline \\[-11pt]
Diagram category & Curried algebra \\[2pt]
\hline \\[-11pt]
Brauer & Symplectic Lie algebra on $\bV \oplus \bV^*$ \\
Signed Brauer & Orthogonal Lie algebra on $\bV \oplus \bV^*$ \\
Spin Brauer &  Orthosymplectic Lie superalgebra on $\bV[1] \oplus \bk \oplus \bV^*[1]$\\
Signed spin Brauer & Orthogonal Lie algebra on $\bV \oplus \bk \oplus \bV^*$\\
Periplectic Brauer & Periplectic Lie superalgebra on $\bV \oplus \bV^*[1]$ \\
Partition & Weyl Lie algebra on $\bV \oplus \bV^*$ \\
Degenerate partition & Hamiltonian Lie algebra on $\bV \oplus \bV^*$ \\
$\FI\sharp(\delta)$ & Heisenberg Lie algebra on $\bV \oplus \bV^*$\\
$\FI$ & Symmetric algebra on $\bV$ \\
$\FA$ & Witt Lie algebra on $\bV^*$ \\
$\FA^{\op}$ & Witt Lie algebra on $\bV$ \\[2pt]
\thickhline
\end{tabular}
\caption{Diagram categories and corresponding curried algebras in $\Mod_{\FB}$.} \label{fig1}
\end{figure}

\subsection{Uses}

There are a few reasons that the curried perspective on diagram categories is useful. First, it provides intuition: e.g., knowing that $\FI\sharp$-modules are modules for a Heisenberg algebra can help one guess how they should behave (though for $\FI\sharp$ itself this is not really necessary, since they are well understood). Second, it suggests new directions: for example, the curried Hamiltonian Lie algebra led us to a novel variant of the partition category that we expect to be interesting.

Finally, the curried perspective helps in applying Schur--Weyl duality as in \cite{infrank} (and this was our main motivation). For us, Schur--Weyl duality is the statement that, in characteristic~0, the category $\Mod_{\FB}$ is equivalent to the category $\Rep^{\pol}(\GL)$ of polynomial representations of the infinite general linear group. This equivalence is a tensor equivalence, so anything stated using the tensor structure on $\Mod_{\FB}$ will transfer nicely to $\Rep^{\pol}(\GL)$. Using this, we find that the Schur--Weyl dual of a module for the Brauer category belongs to parabolic category $\cO$ for an infinite rank symplectic Lie algebra. Furthermore, due to the existence of specialization functors from $\Rep^\pol(\GL)$ to $\Rep^\pol(\GL_n)$ for all finite $n$, we immediately get specialization functors from this parabolic category $\cO$ in the infinite rank case to the finite rank case. This will be the focus of the next paper in this series.

\subsection{Method}

Establishing an equivalence between a curried algebra and a diagram category is entirely elementary, but it can get somewhat complicated. We have therefore developed the following method to treat this problem systematically and keep different concerns isolated:
\begin{enumerate}
\item We first carry out the currying process. We start with a ``model algebra'' $A$ built out of a vector space and its dual, and write down exactly what an $A$-module is without using duals, by the currying procedure. We extrapolate from this a general definition of curried $A$-module in a tensor category.
\item We then specialize this notion to the tensor category $\Mod_{\FB}$, and write down exactly what a curried $A$-module is in terms of $\FB$-modules equipped with certain operations.
\item Finally, we match the above description to a diagram category; this typically involves finding a presentation for the diagram category.
\end{enumerate}
Here is how this process works for relating the symplectic Lie algebra and the Brauer category:
\begin{enumerate}
\item Let $V$ be a finite dimensional vector space. Then $V \oplus V^*$ carries a canonical symplectic form. We take $\fsp(V \oplus V^*)$ to be our model algebra. We have a natural decomposition
\begin{displaymath}
\fsp(V \oplus V^*) = \Div^2(V^*) \oplus \fgl(V) \oplus \Div^2(V).
\end{displaymath}
We thus see that giving a $\fsp(V \oplus V^*)$-module $M$ amounts to giving linear maps
\begin{displaymath}
a \colon V \otimes M \to V \otimes M, \qquad b \colon \Div^2(V) \otimes M \to M, \qquad b' \colon M \to \Sym^2(V) \otimes M.
\end{displaymath}
satisfying certain conditions, which we determine explicitly. Given an object $V$ in a tensor category, we define a module for the \defi{curried symplectic algebra} $\ul{\fsp}(V \oplus V^*)$ to be an object $M$ with maps as above satisfying the conditions we just alluded to.
\item We now examine the curried symplectic algebra in linear species. Thus suppose that $M$ is a $\ul{\fsp}(\bV \oplus \bV^*)$-module, where $\bV$ is the standard $\FB$-module. Giving the map $b$ amounts to giving natural maps $\beta \colon M(S \setminus \{i,j\}) \to M(S)$, where $S$ is a finite set and $i$ and $j$ are distinct elements of it; this is what we mean by an operation on the $\FB$-module $M$. We can similarly describe $a$ and $b'$ in terms of operations. We explicitly write down the conditions on these operations that correspond to the defining conditions of $\ul{\fsp}(\bV \oplus \bV^*)$.
\item Finally, we show that an $\FB$-module with operations as above is the same thing as a module for the Brauer category. The basic idea is that $\beta$ gives the action of a single cap, while the operation corresponding to $b'$ gives the action of a single cup. To prove this, one must show that the identities from the previous step give all the defining relations between cups and caps in the Brauer category, which we do.
\end{enumerate}

\subsection{Open problems}

One broad class of open problem is to examine the currying procedure in other situations. There are other algebras that could be interesting to curry, such as the exceptional Lie algebras (see \cite{jun} for some work on $\fg_2$ and $\fe_6$), quantum groups, or truncated Cartan algebras in positive characteristic. Similarly, there are some diagram categories that would be interesting to see from the curried perspective, such as the simplex category, the category $\OS^{\op}$ from \cite[\S 8]{grobner}, or various linear analogues of $\FI$ like $\VI$ or $\VIC$. Finally, we have mainly focused on currying in the tensor category $\Mod_{\FB}$. What about other categories? (See Remark~\ref{rmk:VA} for a comment on $\Mod_{\VB}$.)

In \S \ref{s:abs}, we define an abstract notion of curried algebra. It would be helpful if this notion were better developed. In particular, is there a way to obtain the results of this paper with less casework?

\subsection{Relation to other papers in this series}

This paper can be read independently of the other papers in this series. The following paper \cite{brauercat3} will make essential use of this paper. Many more examples of curried algebras can be found in \cite{curried}.

There is an extensive literature related to Brauer categories, see
\cite{brauercat1} for a more detailed discussion.

\subsection{Notation and conventions}

Throughout, $\bk$ denotes a fixed commutative ring. Unless otherwise stated, a tensor category is a $\bk$-linear category with a $\bk$-bilinear symmetric monoidal structure.

\subsection{Outline}

In \S \ref{s:bg}, we review linear species, and in \S \ref{s:tri}, we review the theory of triangular categories. In \S \ref{s:gl}, we look at the general linear Lie algebra from the curried perspective; this is the most important example. In \S \ref{s:symp}, \S \ref{s:witt}, and \S \ref{s:weyl}, we look at the symplectic, Witt, and Weyl Lie algebras from the curried perspective. These examples are important cases, and also representative of the currying process in general. Finally, in \S \ref{s:abs}, we make a few comments on abstract curried algebras.

\subsection*{Acknowledgments}

Some of the ideas in \S \ref{ss:fa} came out of joint discussions with Phil Tosteson; we thank him for letting us include this material here.

\section{Linear species} \label{s:bg}

\subsection{$\FB$-modules}

Let $\FB$ be the category of finite sets and bijections. An \defi{$\FB$-module} (also called a \defi{linear species}) is a functor $\FB \to \Mod_{\bk}$. A \defi{morphism} (or \defi{map}) between $\FB$-modules is a natural transformation of functors. We let $\Mod_{\FB}$ be the category of $\FB$-modules. It is a Grothendieck abelian category. Note that $\FB$-modules are equivalent to sequences $(M_n)_{n \ge 0}$ where $M_n$ is a representation of the symmetric group $\fS_n$.

Given two $\FB$-modules $M$ and $N$, we define their tensor product by
\begin{displaymath}
(M \otimes N)(S) = \bigoplus_{T \subseteq S} M(T) \otimes N(S \setminus T).
\end{displaymath}
From the sequence point of view, we have
\begin{displaymath}
(M \otimes N)_n = \bigoplus_{i+j=n} \Ind_{\fS_i \times \fS_j}^{\fS_n}(M_i \otimes N_j).
\end{displaymath}
The above tensor product gives $\Mod_{\FB}$ the structure of a symmetric monoidal category.

We define the \defi{standard $\FB$-module}, denoted $\bV$, to be the $\FB$-module that is $\bk$ on sets of cardinality~1, and 0~on all other sets. If $S$ is a finite set of cardinality $n$ then $\bV^{\otimes n}(S)$ is the $\bk$-vector space with basis given by all total orderings $(s_1,\dots,s_n)$ of the elements of $S$, and $\bV^{\otimes n}(T) = 0$ if $|T| \ne n$. There is an additional action of $\sigma \in \fS_n$ on $\bV^{\otimes n}$ given by
\[
  \sigma \cdot (s_1 ,\dots,s_n) = (s_{\sigma^{-1}(1)}, \dots, s_{\sigma^{-1}(n)}).
  \]

  The $n$th symmetric power $\Sym^n(\bV)$ is the $\fS_n$-coinvariants of $\bV^{\otimes n}$. From the above description, we see that this $\FB$-module is 1-dimensional when evaluated on a set $S$ of cardinality $n$; we write $t^S$ for the distinguished basis vector. The symmetric algebra is
  \[
    \Sym(\bV)=\bigoplus_{n \ge 0} \Sym^n(\bV).
  \]
  It admits both a multiplication map
  \[
    m \colon \Sym(\bV) \otimes \Sym(\bV) \to \Sym(\bV),
  \]
  and a comultiplication map
\begin{displaymath}
\Delta \colon \Sym(\bV) \to \Sym(\bV) \otimes \Sym(\bV).
\end{displaymath}
In terms of bases, these maps are given by
\begin{displaymath}
m(t^A \otimes t^B)=t^{A \cup B}, \qquad
\Delta(t^S) = \sum_{S = A \sqcup B} t^A \otimes t^B,
\end{displaymath}
where the second sum is over all decompositions of $S$ as a union of two disjoint subsets.

We can also consider the $n$th divided power $\Div^n(\bV)$, which is the $\fS_n$-invariants of $\bV^{\otimes n}$. Again, on a finite set $S$ of cardinality $n$, this space is 1-dimensional, and we let $t^{[S]}$ be a basis vector. There is an averaging map
\begin{displaymath}
\avg \colon \Sym^n(\bV) \to \Div^n(\bV).
\end{displaymath}
On basis vectors, this takes $t^S$ to $t^{[S]}$, and so it is an isomorphism. This isomorphism is compatible with the multiplication and comultiplications on $\Div(\bV)=\bigoplus_{n \ge 0} \Div^n(\bV)$. For this reason, we will not really need divided powers in the context of $\FB$-modules.

\begin{remark}
  This contrasts with the standard situation in vector spaces: roughly speaking, this is due to the fact that we are dealing with sets rather than multisets, so that the action of $\fS_n$ on $\bV^{\otimes n}(S)$ is free. All of the complications and differences arise in the standard situation due to the existence of monomials with exponents  greater than 1.
\end{remark}

\subsection{Operations on $\FB$-modules} \label{ss:fbstruc}

Let $S$ be a finite set. We write $S^{[n]}$ for the subset of $S^n$ consisting of tuples with distinct coordinates. We let $S^{[*]}=\coprod_{n \ge 0} S^{[n]}$. Given $\ul{x} \in S^{[n]}$, we write $S \setminus \ul{x}$ in place of $S \setminus \{x_1, \ldots, x_n\}$. We say that two elements $\ul{x} \in S^{[n]}$ and $\ul{y} \in S^{[m]}$ are \defi{disjoint} if $\{x_1,\dots,x_n\} \cap \{y_1,\dots,y_m\} = \emptyset$.

Let $M$ be an $\FB$-module. An \defi{operation} on $M$ is a rule $\phi$ that assigns to every finite set $S$ and elements $\ul{x},\ul{y} \in S^{[*]}$ a linear map
\begin{displaymath}
\phi^S_{\ul{x}, \ul{y}} \colon M(S \setminus \ul{y}) \to M(S \setminus \ul{x})
\end{displaymath}
that is natural, in the sense that if $i \colon S \to T$ is a bijection then the diagram
\begin{displaymath}
\xymatrix@C=6em{
M(S \setminus \ul{y}) \ar[r]^{\phi^S_{\ul{x},\ul{y}}} \ar[d]_i &
M(S \setminus \ul{x}) \ar[d]^i \\
N(T \setminus i(\ul{y})) \ar[r]^{\phi^T_{i(\ul{x}),i(\ul{y})}} &
N(T \setminus i(\ul{x})) }
\end{displaymath}
commutes. It is useful to picture operations diagrammatically; see Figure~\ref{fig:op}.

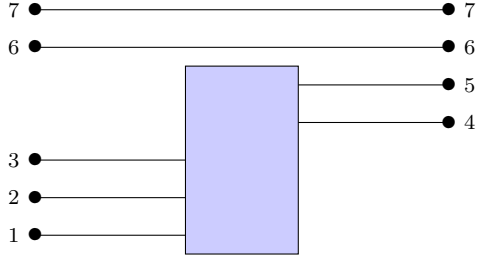
\begin{figure}
\begin{displaymath}
\begin{tikzpicture}[baseline=(current bounding box.center)]
\node at (0,.5) (a1) {$\bullet$};
\node at (0,1) (a2) {$\bullet$};
\node at (0,1.5) (a3) {$\bullet$};
\node at (0,3) (a6) {$\bullet$};
\node at (0,3.5) (a7) {$\bullet$};
\node at (5.5,2) (b4) {$\bullet$};
\node at (5.5,2.5) (b5) {$\bullet$};
\node at (5.5,3) (b6) {$\bullet$};
\node at (5.5,3.5) (b7) {$\bullet$};
\node[xshift=-8pt] at (a1.center) {\tiny 1};
\node[xshift=-8pt] at (a2.center) {\tiny 2};
\node[xshift=-8pt] at (a3.center) {\tiny 3};
\node[xshift=-8pt] at (a6.center) {\tiny 6};
\node[xshift=-8pt] at (a7.center) {\tiny 7};
\node[xshift=8pt] at (b4.center) {\tiny 4};
\node[xshift=8pt] at (b5.center) {\tiny 5};
\node[xshift=8pt] at (b6.center) {\tiny 6};
\node[xshift=8pt] at (b7.center) {\tiny 7};
\draw (a1.center) to (2,.5);
\draw (a2.center) to (2,1);
\draw (a3.center) to (2,1.5);
\draw (b4.center) to (3.5,2);
\draw (b5.center) to (3.5,2.5);
\filldraw [fill=blue!20!white, draw=black] (2,.25) rectangle (3.5,2.75);
\draw (a6.center) to (b6.center);
\draw (a7.center) to (b7.center);
\end{tikzpicture}
\end{displaymath}
\caption{Diagrammatic view of $\phi^S_{\ul{x},\ul{y}}$ where $S=\{1,\ldots,7\}$, $\ul{x}=(1,2,3)$, and $\ul{y}=(4,5)$. We picture the operation as the box that takes input on the $\ul{x}$ strands and produces output on the $\ul{y}$ strands.}
\label{fig:op}
\end{figure}

The definition of operation is quite general; in practice, our operations will be constrained in various ways. We mention a few of the important constraints here. Fix an operation $\phi$ for what follows.
\begin{itemize}
\item We say that $\phi$ is \defi{symmetric} if $\phi^S_{\ul{x},\ul{y}}$ is invariant under permutations of $\ul{x}$ and $\ul{y}$. In this case, we can simply regard $\ul{x}$ and $\ul{y}$ as subsets $A$ and $B$ of $S$, and we typically write $\phi^S_{A,B}$ instead.
\item Similarly, we say that $\phi$ is \defi{skew-symmetric} if $\phi^S_{\ul{x},\ul{y}}$ transforms under the sign character when $\ul{x}$ or $\ul{y}$ is permuted.
\item We say that $\phi$ is an \defi{$(m,n)$-operation} if $\phi^S_{\ul{x},\ul{y}}=0$ unless $\ul{x}$ has length $m$ and $\ul{y}$ has length $n$. In this case, we typically regard $\phi^S_{\ul{x},\ul{y}}$ as only defined on such tuples.
\item We say that $\phi$ is \defi{simple} if $\phi^S_{\ul{x},\ul{y}}=0$ unless $\ul{x}$ and $\ul{y}$ are disjoint. Again, in this case we typically regard $\phi^S_{\ul{x},\ul{y}}$ as only being defined on disjoint tuples.
\end{itemize}
Every operation can be expressed in terms of simple operations. We explain this in the case where $\phi$ is symmetric, as this somewhat simplifies the situation. For $n \in \bN$ define a simple operation $\phi[n]$ by $\phi[n]^S_{A,B}=\phi^{S \amalg [n]}_{A \amalg [n], B \amalg [n]}$ if $A$ and $B$ are disjoint. The naturality of $\phi$ implies that
\begin{displaymath}
\phi^S_{A,B}=\phi[n]^{S \setminus (A \cap B)}_{A \setminus B, B \setminus A}
\end{displaymath}
where $n=\#(A \cap B)$. Thus $\phi$ determines, and is determined by, the sequence of simple operations $(\phi[n])_{n \ge 0}$.

Operations are closely related to the tensor product on $\FB$-modules. For example, giving a symmetric $(m,n)$-operation $\phi$ on $M$ is equivalent to giving a map of $\FB$-modules
\begin{displaymath}
a \colon \Sym^n(\bV) \otimes M \to \Sym^m(\bV) \otimes M.
\end{displaymath}
Indeed, given a finite set $S$, a subset $B$ of $S$ of cardinality $n$, and an element $x \in M(S \setminus B)$, we can write
\begin{displaymath}
a(t^B \otimes x) = \sum_{\substack{A \subseteq S\\ \# A=m}} t^A \otimes \phi^S_{A,B}(x)
\end{displaymath}
where $\phi^S_{A,B}(x)$ belongs to $M(S \setminus A)$. This defines a map
\begin{displaymath}
\phi^S_{A,B} \colon M(S \setminus B) \to M(S \setminus A),
\end{displaymath}
and these maps define an $(m,n)$-operation $\phi$.

Let $\phi$ and $\psi$ be operations. We say that $\phi$ and $\psi$ \defi{commute} if the following condition holds: given a finite set $S$ and tuples $\ul{x}, \ul{y}, \ul{w}, \ul{z} \in S^{[*]}$ such that $\ul{x}$ and $\ul{w}$ are disjoint and $\ul{y}$ and $\ul{z}$ are disjoint, the diagram
\begin{displaymath}
\xymatrix@C=8em{
M(S \setminus (\ul{y} \cup \ul{z})) \ar[r]^{\phi^{S \setminus \ul{z}}_{\ul{x},\ul{y}}} \ar[d]_{\psi^{S \setminus \ul{y}}_{\ul{w},\ul{z}}} &
M(S \setminus (\ul{x} \cup \ul{z})) \ar[d]^{\psi^{S \setminus \ul{x}}_{\ul{w},\ul{z}}} \\
M(S \setminus (\ul{y} \cup \ul{w})) \ar[r]^{\phi^{S \setminus \ul{w}}_{\ul{x},\ul{y}}} &
M(S \setminus (\ul{x} \cup \ul{w})) }
\end{displaymath}
commutes. See Figure~\ref{fig:comm-op} for a diagrammatic interpretation of this condition. Similarly, we say that $\phi$ and $\psi$ \defi{skew-commute} if the two paths above are negatives of each other. We note that an operation need not commute with itself, and can skew-commute with itself while still being non-trivial (even in characteristic~0).

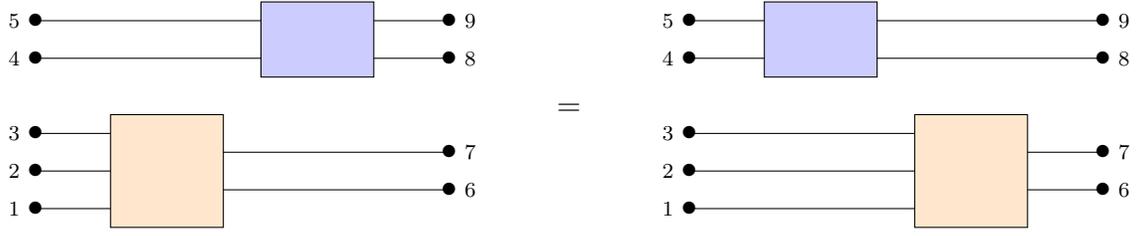
\begin{figure}
\begin{displaymath}
\begin{tikzpicture}[baseline=(current bounding box.center)]
\node at (0,.5) (a1) {$\bullet$};
\node at (0,1) (a2) {$\bullet$};
\node at (0,1.5) (a3) {$\bullet$};
\draw (a1.center) to (1,.5);
\draw (a2.center) to (1,1);
\draw (a3.center) to (1,1.5);
\filldraw [fill=orange!20!white, draw=black] (1,.25) rectangle (2.5,1.75);
\node at (0,2.5) (a4) {$\bullet$};
\node at (0,3) (a5) {$\bullet$};
\draw (a4.center) to (3.5,2.5);
\draw (a5.center) to (3.5,3);
\filldraw [fill=blue!20!white, draw=black] (3,2.25) rectangle (4.5,3.25);
\node at (5.5,.75) (b1) {$\bullet$};
\node at (5.5,1.25) (b2) {$\bullet$};
\node at (5.5,2.5) (b3) {$\bullet$};
\node at (5.5,3) (b4) {$\bullet$};
\draw (b1.center) to (2.5,.75);
\draw (b2.center) to (2.5,1.25);
\draw (b3.center) to (4.5,2.5);
\draw (b4.center) to (4.5,3);
\node[xshift=-8pt] at (a1.center) {\tiny 1};
\node[xshift=-8pt] at (a2.center) {\tiny 2};
\node[xshift=-8pt] at (a3.center) {\tiny 3};
\node[xshift=-8pt] at (a4.center) {\tiny 4};
\node[xshift=-8pt] at (a5.center) {\tiny 5};
\node[xshift=8pt] at (b1.center) {\tiny 6};
\node[xshift=8pt] at (b2.center) {\tiny 7};
\node[xshift=8pt] at (b3.center) {\tiny 8};
\node[xshift=8pt] at (b4.center) {\tiny 9};
\end{tikzpicture}
\qquad=\qquad
\begin{tikzpicture}[baseline=(current bounding box.center)]
\node at (0,.5) (a1) {$\bullet$};
\node at (0,1) (a2) {$\bullet$};
\node at (0,1.5) (a3) {$\bullet$};
\draw (a1.center) to (3,.5);
\draw (a2.center) to (3,1);
\draw (a3.center) to (3,1.5);
\filldraw [fill=orange!20!white, draw=black] (3,.25) rectangle (4.5,1.75);
\node at (0,2.5) (a4) {$\bullet$};
\node at (0,3) (a5) {$\bullet$};
\draw (a4.center) to (1,2.5);
\draw (a5.center) to (1,3);
\filldraw [fill=blue!20!white, draw=black] (1,2.25) rectangle (2.5,3.25);
\node at (5.5,.75) (b1) {$\bullet$};
\node at (5.5,1.25) (b2) {$\bullet$};
\node at (5.5,2.5) (b3) {$\bullet$};
\node at (5.5,3) (b4) {$\bullet$};
\draw (b1.center) to (4.5,.75);
\draw (b2.center) to (4.5,1.25);
\draw (b3.center) to (2.5,2.5);
\draw (b4.center) to (2.5,3);
\node[xshift=-8pt] at (a1.center) {\tiny 1};
\node[xshift=-8pt] at (a2.center) {\tiny 2};
\node[xshift=-8pt] at (a3.center) {\tiny 3};
\node[xshift=-8pt] at (a4.center) {\tiny 4};
\node[xshift=-8pt] at (a5.center) {\tiny 5};
\node[xshift=8pt] at (b1.center) {\tiny 6};
\node[xshift=8pt] at (b2.center) {\tiny 7};
\node[xshift=8pt] at (b3.center) {\tiny 8};
\node[xshift=8pt] at (b4.center) {\tiny 9};
\end{tikzpicture}
\end{displaymath}
\caption{Commuting operations.}
\label{fig:comm-op}
\end{figure}

\section{Triangular categories} \label{s:tri}

Most of the diagram categories considered in this paper are \emph{triangular categories}, a notion introduced in \cite{brauercat1} (and similar to the notion of semi-infinite highest weight category in the sense of \cite{BrundanStroppel}). We will use this structure to aid us in establishing presentations for these categories. We recall the definition here and establish a few properties of these categories that will be useful.

Let $\fG$ be a $\bk$-linear category satisfying the following condition:
\begin{itemize}
\item[(T0)] The category $\fG$ is essentially small, and all $\Hom$ spaces are finite dimensional.
\end{itemize}
We denote the set of isomorphism classes in $\fG$ by $\vert \fG \vert$. Recall that a subcategory is \defi{wide} if it contains all objects.

\begin{definition}
A \defi{triangular structure} on $\fG$ is a pair $(\fU, \fD)$ of wide subcategories of $\fG$ such that the following axioms hold:
\begin{itemize}
\item[(T1)] We have $\End_{\fU}(x)=\End_{\fD}(x)$ for all objects $x$.
\item[(T2)] There exists a partial order $\le$ on $\vert \fG \vert$ such that:
\begin{enumerate}
\item For all $x \in \vert \fG \vert$ there are only finitely many $y \in \vert \fG \vert$ with $y \le x$.
\item The category $\fU$ is upwards with respect to $\le$, i.e., if there exists a non-zero morphism $x \to y$, then $x \le y$.
\item The category $\fD$ is downwards with respect to $\le$, i.e., if there exists a non-zero morphism $x \to y$, then $y \le x$.
\end{enumerate}
\item[(T3)] For all $x,y \in \fG$, the natural map
\begin{displaymath}
\bigoplus_{y \in \vert \fG \vert} \Hom_{\fU}(y,z) \otimes_{\End_{\fU}(y)} \Hom_{\fD}(x,y) \to \Hom_{\fG}(x,z)
\end{displaymath}
is an isomorphism.
\end{itemize}
A \defi{triangular category} is a $\bk$-linear category satisfying (T0) equipped with a triangular structure.
\end{definition}

\begin{remark}
In \cite{brauercat1}, we required the rings $\End_{\fU}(x)$ to be semisimple; we do not make that assumption here.
\end{remark}

Fix a triangular category $\fG$, and set
\[
  \fM = \fU \cap \fD.
\]
Note that all non-zero morphisms in $\fM$ are between isomorphic objects; in our applications $\fM$ will almost always be the $\bk$-linearization of $\FB$. Recall that if $\fC$ is a $\bk$-linear category then a \defi{$\fC$-module} is a $\bk$-linear functor $\fC \to \Mod_{\bk}$. We are interested in modules over the categories $\fG$, $\fU$, $\fD$, and $\fM$. Suppose that $\fC$ is one of these categories. Then $\fC$ has the same objects of $\fM$ and contains $\fM$. Thus a $\fC$-module can be regarded as an $\fM$-module equipped with extra structure; we refer to this extra structure as a \defi{$\fC$-structure}. By (T3) it follows that a $\fG$-structure on an $\fM$-module is determined by its restrictions to $\fD$ and $\fU$. We say that a $\fD$-structure and a $\fU$-structure on an $\fM$-module are \defi{compatible} if they come from a $\fG$-structure. We now investigate compatibility in more detail.

Fix an $\fM$-module $M$ equipped with a $\fD$-structure and a $\fU$-structure. Let $\alpha$ be a morphism in $\fG$.  Write
\begin{displaymath}
\alpha = \sum_{i=1}^n \phi_i \circ \psi_i
\end{displaymath}
with $\phi_i$ in $\fU$ and $\psi_i$ in $\fD$, which is possible by (T3). We then define
\begin{displaymath}
\alpha_* = \sum_{i=1}^n (\phi_i)_* (\psi_i)_*.
\end{displaymath}
This is well-defined by (T3) and the fact that the $\fU$- and $\fD$-structures agree on $\fM$. Suppose that $\beta$ is a second morphism such that $\beta \circ \alpha$ is defined. We say that $(\alpha, \beta)$ is \defi{compatible} if $(\beta \circ \alpha)_*=\beta_* \alpha_*$. We note that $(\alpha, \beta)$ is automatically compatible if $\beta$ belongs to $\fU$, or if $\alpha$ belongs to $\fD$. Clearly, the $\fU$- and $\fD$-structures on $M$ are compatible if and only if $(\alpha, \beta)$ is compatible for all $\alpha$, $\beta$ such that $\beta \circ \alpha$ is defined. In fact, one has the following:

\begin{proposition}
The $\fU$- and $\fD$-structures on $M$ are compatible if and only if all pairs $(\phi, \psi)$ with $\phi$ in $\fU$ and $\psi$ in $\fD$ are compatible.
\end{proposition}

\begin{proof}
Let $\alpha$ and $\beta$ be morphisms in $\fG$ such that $\beta \circ \alpha$ is defined. Write
\begin{displaymath}
\alpha = \sum_{i=1}^n \phi_i \circ \psi_i, \qquad
\beta = \sum_{j=1}^m \phi_j' \circ \psi_j'
\end{displaymath}
with $\phi_i$ and $\phi'_j$ in $\fU$ and $\psi_i$ and $\psi'_j$ in $\fD$. For each $(i,j)$, write
\begin{displaymath}
\psi_j' \circ \phi_i = \sum_{k=1}^{N_{i,j}} \phi''_{i,j,k} \circ \psi''_{i,j,k}
\end{displaymath}
where, again, the $\phi''$ belong to $\fU$ and the $\psi''$ belong to $\fD$. Then
\begin{displaymath}
\beta \circ \alpha= \sum_{i=1}^n \sum_{j=1}^m \sum_{k=1}^{N_{i,j}} (\phi'_j \circ \phi''_{i,j,k}) \circ (\psi''_{i,j,k} \circ \psi_i).
\end{displaymath}
We thus have
\begin{align*}
(\beta \circ \alpha)_*
&= \sum_{i,j,k} (\phi'_j)_* (\phi''_{i,j,k})_* (\psi''_{i,j,k})_* (\psi_i)_* \\
&= \sum_{i,j} (\phi'_j)_* (\psi'_j)_* (\phi_i)_* (\psi_i)_* \\
&= \beta_* \alpha_*,
\end{align*}
where in the first step we used the definition of $(\beta \circ \alpha)_*$, in the second step we used the compatibility of $(\phi_i, \psi'_j)$ for all $i$ and $j$, and in the third step we used the definitions of $\alpha_*$ and $\beta_*$. Thus $(\alpha, \beta)$ is compatible, and the proof is complete.
\end{proof}

We now give a refinement of the above criterion. Let $\fC$ be a $\bk$-linear category. We say a class of morphisms $\cC$ in $\fC$ \defi{generates} if every morphism in $\fC$ can be expressed as a $\bk$-linear combination of finite compositions of morphisms in $\cC$.

\begin{proposition} \label{prop:tri-comp}
Let $\cU$ generate $\fU$ and let $\cD$ generate $\fD$. Suppose that $(\phi, \psi)$ are compatible whenever $\phi \in \cU$ and $\psi \in \cD$, and $\psi \circ \phi$ is defined. Then the $\fU$- and $\fD$-structures are compatible.
\end{proposition}

\begin{proof}
We write $s(\phi)$ and $t(\phi)$ for the source and target of a morphism $\phi$. For $x \in \vert \fG \vert$, consider the following statements:
\begin{itemize}
\item[$S(x)$:] Let $\alpha$ and $\beta$ be morphisms in $\fG$ such that $\beta \circ \alpha$ is defined and $t(\alpha) =x$. Then $(\alpha, \beta)$ is compatible.
\item[$S_{\le x}$:] Statement $S(y)$ holds for all $y \le x$.
\item[$S_{<x}$:] Statement $S(y)$ holds for all $y<x$.
\end{itemize}
Clearly, it suffices to prove $S(x)$ for all $x$. We prove that $S_{<x}$ implies $S_{\le x}$ for all $x$. This implies $S(x)$ for all $x$ by an inductive argument, which is enabled by the condition (T2a). Thus let $x \in \vert \fG \vert$ be given and suppose $S_{<x}$ holds.

First suppose that $\phi$ is a morphism in $\fU$ and $\psi$ is a morphism in $\fD$ such that $\psi \circ \phi$ is defined and $t(\phi) \le x$. We show that $(\phi, \psi)$ is compatible. If $\phi$ or $\psi$ belongs to $\fM$, the statement is trivial, so assume this is not the case. We can express $\phi$ as a linear combination of compositions of morphisms in $\cU$. Since compatibility interacts well with linear combinations, it suffices to treat the case where $\phi$ is a composition of morphisms in $\cU$. We can thus write $\phi=\phi^1 \phi^2$ where $\phi^1$ belongs to $\cU$ but not to $\fM$, and $\phi^2$ belongs to $\fU$. Similarly, we can assume $\psi=\psi^2 \psi^1$ where $\psi^1$ belongs to $\cD$ but not to $\fM$, and $\psi^2$ belongs to $\fD$. Write
\begin{displaymath}
\psi^1 \circ \phi^1 = \sum_{i=1}^n \phi_i^3 \circ \psi_i^3
\end{displaymath}
with $\phi_i^3$ in $\fU$ and $\psi_i^3$ in $\fD$. Since $(\phi^1, \psi^1)$ is compatible by assumption, we have
\begin{displaymath}
\psi^1_* \phi^1_* = (\psi^1 \circ \phi^1)_* = \sum_{i=1}^n (\phi_i^3)_* (\psi_i^3)_*.
\end{displaymath}
We thus have
\begin{align*}
\psi^2_* \psi^1_* \phi^1_* \phi^2_*
&= \sum_{i=1}^n \psi^2_* (\phi_i^3)_*(\psi_i^3)_* \phi^2_*
= \sum_{i=1}^n \psi^2_* (\phi_i^3)_* (\psi_i^3 \circ \phi^2)_* \\
&= \sum_{i=1}^n \psi^2_* (\phi_i^3 \circ \psi_i^3 \circ \phi^2)_*
= \sum_{i=1}^n (\psi^2 \circ \phi_i^3 \circ \psi_i^3 \circ \phi^2)_*
= (\psi^2 \circ \psi^1 \circ \phi^1 \circ \phi^2)_*
\end{align*}
where we have repeatedly used $S_{<x}$. Note that
\begin{align*}
t(\phi^2)=s(\phi^1)<t(\phi^1) &\le x \\
t(\psi_i^3) \le s(\psi_i^3)=t(\phi^2) &<x \\
t(\phi_i^3)=t(\psi^1)<s(\psi^1)=t(\phi^1) &\le x
\end{align*}
which justifies applying $S_{<x}$ in each case. We thus see that $\psi_* \phi_*=(\psi \circ \phi)_*$, and so $(\phi, \psi)$ is compatible.

We now treat the general case. Thus let $\alpha$ and $\beta$ be morphisms in $\fG$ such that $\beta \circ \alpha$ is defined and $t(\alpha) \le x$. We show that $(\alpha, \beta)$ is compatible. Write
\begin{displaymath}
\alpha = \sum_{i=1}^n \phi_i \circ \psi_i, \qquad
\beta = \sum_{j=1}^m \phi'_j \circ \psi_j'
\end{displaymath}
where $\phi_i$ and $\phi'_j$ belong to $\fU$ and the $\psi_i$ and $\psi'_j$ belong to $\fD$.
We have
\begin{align*}
\beta_* \alpha_*
&= \sum_{i,j} (\phi'_j)_* (\psi'_j)_* (\phi_i)_* (\psi_i)_*
=\sum_{i,j} (\phi'_j)_* (\psi'_j \circ \phi_i)_* (\psi_i)_* \\
&= \sum_{i,j} (\phi'_j \circ \psi'_j \circ \phi_i \circ \psi_i)_*
= (\beta \circ \alpha)_*.
\end{align*}
In the second step we used the previous paragraph, and in the third step we used the automatic compatibility for morphisms in $\fU$ and $\fD$. This completes the proof.
\end{proof}

\section{The general linear Lie algebra} \label{s:gl}

\subsection{Currying}

Let $V$ be a finite-dimensional vector space, and consider the Lie algebra $\fgl(V)$. A \textbf{representation} of $\fgl(V)$ consists of a vector space $M$ equipped with a linear map
\begin{displaymath}
\mu \colon \fgl(V) \to \End(M)
\end{displaymath}
such that
\begin{equation} \label{eq:rep}
\mu([X,Y])=[\mu(X),\mu(Y)]
\end{equation}
holds for all $X,Y \in \fgl(V)$, where $[X,Y]=XY-YX$ denotes the commutator. Now, $\fgl(V)$ is canonically isomorphic to $V \otimes V^*$. Thus giving a linear map $\mu$ as above is equivalent to giving a linear map
\begin{displaymath}
a \colon V \otimes M \to V \otimes M,
\end{displaymath}
and the following proposition determines the condition that \eqref{eq:rep} imposes on $a$. We first introduce some notation. For a linear map $a$ as above, we define maps
\begin{displaymath}
a_1, a_2, \tau \colon V \otimes V \otimes M \to V \otimes V \otimes M
\end{displaymath}
as follows. First, $\tau$ switches the first two tensor factors, i.e., $\tau(v \otimes w \otimes x)=w \otimes v \otimes x$. Next, $a_2$ is $\id \otimes a$, i.e., $a_2(v \otimes w \otimes x) = v \otimes a(w \otimes x)$. Finally, $a_1 = \tau \circ a_2 \circ \tau$. We now have:

\begin{proposition} \label{prop:glid}
Let $\mu$ and $a$ be corresponding linear maps as above. Then $\mu$ satisfies \eqref{eq:rep} if and only if $a$ satisfies the equation
\addtocounter{equation}{-1}
\begin{subequations}
\begin{equation} \label{eq:glid}
[a_1,a_2]=\tau(a_1-a_2).
\end{equation}
\end{subequations}
\end{proposition}

\begin{proof}
Assume $\mu$ defines a representation of $\fgl(V)$. Let $\{v_i\}_{1 \le i \le n}$ be a basis for $V$, let $\{v_i^*\}$ be the dual basis, and write $v_iv_j^*$ for the element of $\fgl(V)$ corresponding to $v_i \otimes v_j^*$. The map $a$ is given by
\begin{displaymath}
a(v_i \otimes x) = \sum_{j=1}^n v_j \otimes (v_iv_j^*) x,
\end{displaymath}
where here $(v_iv_j^*) x$ denotes $\mu(v_iv_j^*)(x)$. We have
\begin{displaymath}
a_1(a_2(v_i \otimes v_k \otimes x)) = \sum_{1 \le j,\ell \le n} v_j \otimes v_{\ell} \otimes (v_iv_j^*) (v_kv_{\ell}^*) x.
\end{displaymath}
The formula for $a_2(a_1(v_i \otimes v_k \otimes x))$ is the same, except that the order of $v_iv_j^*$ and $v_kv_{\ell}^*$ on the right is reversed. We thus find
\begin{displaymath}
[a_1, a_2](v_i \otimes v_k \otimes x) = \sum_{1 \le j,\ell \le n} v_j \otimes v_{\ell} \otimes [v_iv_j^*,v_kv_{\ell}^*] x.
\end{displaymath}
Using the formula
\begin{displaymath}
[v_iv_j^*,v_kv_{\ell}^*] = \delta_{j,k} (v_iv_{\ell}^*) - \delta_{i,\ell} (v_kv_j^*),
\end{displaymath}
we find
\begin{displaymath}
[a_1, a_2](v_i \otimes v_k \otimes x) = \left( \sum_{1 \le \ell \le n} v_k \otimes v_{\ell} \otimes (v_iv_{\ell}^*) x \right) - \left( \sum_{1 \le j \le n} v_j \otimes v_i \otimes (v_kv_j^*) x \right).
\end{displaymath}
The first term on the right is $(\tau a_1)(v_i \otimes v_k \otimes x)$, while the second is $(\tau a_2)(v_i \otimes v_k \otimes x)$. We thus see that $a$ satisfies \eqref{eq:glid}. The reasoning is reversible, and so if $a$ satisfies \eqref{eq:glid} then $\mu$ defines a representation of $\fgl(V)$.
\end{proof}

\begin{remark}
The identity \eqref{eq:glid} can be expressed equivalently in the form
\begin{displaymath}
\tau a \tau a - a \tau a \tau = a \tau - \tau a,
\end{displaymath}
where here we have written $a$ in place of $a_2=\id_V \otimes a$.
\end{remark}

We now extrapolate a general definition from Proposition~\ref{prop:glid}:

\begin{definition}
Let $V$ be an object of a tensor category $\cC$. We define the \textbf{curried general linear Lie algebra} on $V$, denoted $\ul{\fgl}(V)$, as follows. A representation of $\ul{\fgl}(V)$ consists of an object $M$ of $\cC$ together with a morphism $a \colon V \otimes M \to V \otimes M$ such that the equation $[a_1,a_2]=\tau (a_1-a_2)$ holds, using notation as in Proposition~\ref{prop:glid}.
\end{definition}

If $M,N$ are $\ul{\fgl}(V)$-module, then a \textbf{morphism} of $\ul{\fgl}(V)$-modules $\phi \colon M \to N$ is a morphism in $\cC$ such that the diagram
\begin{displaymath}
\xymatrix{
V \otimes M \ar[d]_{\id \otimes \phi} \ar[r]^a & V \otimes M \ar[d]^{\id \otimes \phi} \\
V \otimes N \ar[r]^a & V \otimes N }
\end{displaymath}
commutes. We write $\Rep(\ul{\fgl}(V))$ for the category of $\ul{\fgl}(V)$-modules. It is easily verified to be an abelian category.

\subsection{General observations} \label{ss:glgen}

We now discuss some basic aspects of $\ul{\fgl}(V)$-modules.

\textit{Trivial representations.} Given any object $M$ of $\cC$, we can define a $\ul{\fgl}(V)$-module structure on $M$ by taking the structure map $V \otimes M \to V \otimes M$ to be the zero map. We refer to this as the \defi{trivial representation} of $\ul{\fgl}(V)$ on $M$. By \emph{the} trivial representation, we mean the one on the unit object $\bone$.

\textit{The standard representation.} Let $M=V$ and take $a=\tau$. We verify \eqref{eq:glid}. This is an identity among endomorphisms of $V^{\otimes 3}$. What is called $\tau$ there is really $\tau_{12}$, and what is called $a$ is $\tau_{23}$. Using the braid relation, we have
\begin{displaymath}
\tau_{12} \tau_{23} \tau_{12} \tau_{23} = \tau_{23} \tau_{12} \tau_{23}^2 = \tau_{23} \tau_{12}.
\end{displaymath}
Similarly, $\tau_{23} \tau_{12} \tau_{23} \tau_{12}=\tau_{12} \tau_{23}$. The identity \eqref{eq:glid} follows. We call $V$ with this action the \defi{standard representation} of $\ul{\fgl}(V)$.

\textit{Tensor products.} Suppose that $M$ and $N$ are two $\ul{\fgl}(V)$-modules, with action maps $a$ and $b$. Regard $\End(V \otimes M)$ and $\End(V \otimes N)$ as subalgebras of $\End(V \otimes M \otimes N)$ in the obvious way. We give $M \otimes N$ the structure of a $\ul{\fgl}(V)$-module by taking the action map to be $a+b$. To see that this satisfies \eqref{eq:glid}, note that $a_1$ and $b_2$ commute in $\End(V^{\otimes 2} \otimes M \otimes N)$, since $a_1$ uses the first and third factors and $b_2$ the second and fourth, and similarly for $b_1$ and $a_2$. Therefore,
\begin{displaymath}
[a_1+b_1,a_2+b_2]
= [a_1,a_2]+[b_1,b_2]+[a_1,b_2]+[b_1,a_2]
= \tau(a_1-a_2)+\tau(b_1-b_2).
\end{displaymath}
The operation $\otimes$ endows $\Rep(\ul{\fgl}(V))$ with the structure of a tensor category.

\textit{Tensor powers of the standard representation.} Let $M=V^{\otimes n}$ be the $n$th tensor power of the standard representation. The action map is the endomorphism $\sum_{i=2}^{n+1} \tau_{1,i}$ of $V^{\otimes (n+1)}$.

\textit{Twisting by trace.} Let $M$ be a $\ul{\fgl}(V)$-module with structure map $a \colon V \otimes M \to V \otimes M$, and let $\delta$ be an element of the coefficient field $\bk$. Then the map $a+\delta \cdot \id_{V \otimes M}$ defines a new $\ul{\fgl}(V)$ representation on $M$. We denote the resulting $\ul{\fgl}(V)$-module by $M(\delta)$. We have $M(\delta)=M \otimes \bone(\delta)$ and $\bone(\delta_1) \otimes \bone(\delta_2) = \bone(\delta_1+\delta_2)$. The representation $\bone(\delta)$ is analogous to the representation of $\fgl_n$ given by $X \mapsto \delta \tr(X)$.

\textit{Behavior under tensor functors.} We have defined $\ul{\fgl}(V)$-modules purely in terms of the tensor structure on $\cA$. It follows that if $\Phi \colon \cA \to \cB$ is a tensor functor (with no exactness properties assumed) then $\Phi$ carries $\ul{\fgl}(V)$-modules to $\ul{\fgl}(\Phi(V))$-modules. This will remain true for the other curried Lie algebras we define, and will be a useful observation later on.

\textit{Some unexpected behavior.} There are examples where an ``actual'' $\fgl(V)$ exists in $\cC$, but where representations of $\fgl(V)$ and $\ul{\fgl}(V)$ are not the same. For example, let $\cC$ be the category of abelian groups, let $V=(\bZ/p\bZ)^n$, and let $M=\bZ^n$. Then $M$ does not admit a non-trivial representation of $\fgl(V)$: since $\End(M)$ is torsion-free under addition, there are no non-zero maps $\fgl(V) \to \End(M)$. However, $M$ does admit a non-trivial representation of $\ul{\fgl}(V)$: one can take the switching of factors map on $V \otimes M \cong V \otimes V$. The source of this discrepancy is that $V$ is not a dualizable object in $\cC$, and so $\fgl(V)$ is not isomorphic to $V \otimes V^*$; the curried algebra $\ul{\fgl}(V)$ always behaves as if it were $V \otimes V^*$.

\subsection{In species} \label{ss:gl-fb}

Let $\bV$ be the standard $\FB$-module and let $M$ be an arbitrary $\FB$-module. Suppose we have a map of $\FB$-modules
\begin{displaymath}
a \colon \bV \otimes M \to \bV \otimes M.
\end{displaymath}
Given a finite set $S$, an element $j \in S$, and an element $x \in M(S \setminus j)$, we can write
\begin{equation} \label{eq:gl-fb}
a(t^j \otimes x) = t^j \otimes \omega^{S \setminus j}(x) + \sum_{i \in S \setminus j} t^i \otimes \alpha^S_{i,j}(x).
\end{equation}
Thus $\omega$ is a $(0,0)$-operation on $M$, i.e., an endomorphism of $M$ as an $\FB$-module, and $\alpha$ is a simple $(1,1)$-operation on $M$.

\begin{proposition} \label{prop:gl-fb}
The map $a$ defines a representation of $\ul{\fgl}(\bV)$ on $M$ if and only if the following conditions hold:
\begin{enumerate}
\item The operations $\alpha$ and $\omega$ commute with themselves and each other.
\item Given a finite set $S$ and three distinct elements $i,j,k \in S$, we have $\alpha^{S \setminus i}_{j,k} \alpha^{S \setminus k}_{i,j} = \alpha^{S \setminus j}_{i,k}$.
\end{enumerate}
\end{proposition}

\begin{proof}
Let $\phi$ be the (non-simple) $(1,1)$-operation on $M$ corresponding to $a$. Thus
\begin{displaymath}
a(t^j \otimes x) = \sum_{i \in S} t^i \otimes \phi^S_{i,j}(x).
\end{displaymath}
We have $\alpha=\phi[0]$ and $\omega=\phi[1]$ in the notation of \S \ref{ss:fbstruc}. Let $S$ be a finite set, let $j,k \in S$ be distinct, and let $x \in M(S \setminus \{j,k\})$. A simple computation gives
\begin{align*}
a_1(a_2(t^j \otimes t^k \otimes x)) &= \sum_{\ell \in S \setminus j} \sum_{i \in S \setminus \ell} t^i \otimes t^{\ell} \otimes \phi^{S \setminus \ell}_{i,j}(\phi^{S \setminus j}_{\ell,k}(x)) \\
a_2(a_1(t^j \otimes t^k \otimes x)) &= \sum_{i \in S \setminus k} \sum_{\ell \in S \setminus i} t^i \otimes t^{\ell} \otimes \phi^{S \setminus i}_{\ell,k}(\phi^{S \setminus k}_{i,j}(x)) \\
\tau a_1(t^j \otimes t^k \otimes x) &= \sum_{\ell \in S \setminus k} t^k \otimes t^{\ell} \otimes \phi^{S \setminus k}_{\ell,j}(x) \\
\tau a_2(t^j \otimes t^k \otimes x) &= \sum_{i \in S \setminus j} t^i \otimes t^j \otimes \phi^{S \setminus j}_{i,k}(x)
\end{align*}
Now, consider the equation $[a_1,a_2]=\tau(a_1-a_2)$. Letting $i,\ell \in S$ be distinct elements and examining the coefficients of $t^i \otimes t^{\ell}$, we obtain the following equations:
\begin{align*}
\phi^{S \setminus \ell}_{i,j} \circ \phi^{S \setminus j}_{\ell,k} &=
\phi^{S \setminus i}_{\ell,k} \circ \phi^{S \setminus k}_{i,j} & \text{if $i \ne k$ and $\ell \ne j$}, \\
\phi^{S \setminus i}_{j,k} \circ \phi^{S \setminus k}_{i,j} &=
\phi^{S \setminus j}_{i,k} & \text{if $i \ne k$ and $\ell=j$}, \\
\phi^{S \setminus \ell}_{k,j} \circ \phi^{S \setminus j}_{\ell,k} &=
\phi^{S \setminus k}_{\ell,j} & \text{if $i=k$ and $\ell \ne j$}, \\
\phi^{S \setminus k}_{j,j} &= \phi^{S \setminus j}_{k,k} & \text{if $i=k$ and $\ell=j$}.
\end{align*}
The first equation above is equivalent to condition (a). The second and third equations above are equivalent to each other, and to (b). The final equation above is automatic: it follows from the naturality of $\phi$. This completes the proof.
\end{proof}

For a finite set $S$ and distinct elements $i,j \in S$, let $\iota^S_{i,j} \colon S \setminus \{j\} \to S \setminus \{i\}$ be the bijection given by
\begin{displaymath}
\iota^S_{i,j}(k) = \begin{cases} j & \text{if $k=i$} \\ k & \text{if $k \ne i$} \end{cases}.
\end{displaymath}
Let $M$ be an $\FB$-module and let $\delta \in \bk$. We define the \defi{$\delta$-standard $\ul{\fgl}(\bV)$-structure} on $M$ to be the representation of $\ul{\fgl}(\bV)$ on $M$ given by  Proposition~\ref{prop:gl-fb} with $\omega=\delta\cdot \id$ and $\alpha^S_{i,j}=(\iota^S_{i,j})_*$. (One easily verifies the conditions of Proposition~\ref{prop:gl-fb}.) Explicitly, for $j \in S$ and $x \in M(S \setminus j)$, we have
\begin{displaymath}
a(t^j \otimes x) = \delta x + \sum_{i \in S \setminus j} t^i \otimes (\iota_{i,j}^S)_*(x).
\end{displaymath}
One easily verifies that this construction is functorial: any map of $\FB$-modules induces a map between the corresponding $\delta$-standard $\ul{\fgl}(\bV)$-modules.

\begin{remark} \label{rmk:gl-seq}
Let $M$ be a $\ul{\fgl}(\bV)$-module. Given a finite set $S$ and an element $i \in S$, define $\rho_i$ to be the composition
\begin{displaymath}
\xymatrix@C=4em{
M(S) \ar[r]^-{\alpha^{S \amalg \{\ast\}}_{i,\ast}} &
M(S \cup \{\ast\} \setminus i) \ar[r]^-{\iota^{S \amalg \{\ast\}}_{\ast,i}} &
M(S), }
\end{displaymath}
where $\{\ast\}$ is a one-point set. One easily verifies that $\rho_i^2=\rho_i$, and that for $i \ne j$, the operators $\rho_i$ and $\rho_j$ commute. Furthermore, for $\pi \in \Aut(S)$ we have $\pi \rho_i \pi^{-1}=\rho_{\pi(i)}$. Let $\fA_n$ be the monoid freely generated by $n$ commuting idempotents. We thus see that $M([n])$ carries a representation of the monoid $\bN \times (\fS_n \ltimes \fA_n)$, where the generator of the $\bN$ acts by the $\omega$ operation. In fact, a $\ul{\fgl}(\bV)$-module $M$ exactly corresponds to a sequence $(M_n)_{n \ge 0}$ where $M_n$ is a representation of $\bN \times (\fS_n \ltimes \fA_n)$. From this point of view, a $\delta$-standard $\ul{\fgl}(\bV)$-module is one where $\fA_n$ acts trivially and the generator of $\bN$ acts by $\delta$.
\end{remark}

\begin{remark}  \label{rmk:pieri}
Assume that $\bk$ is a field of characteristic $0$.  Consider the standard $\ul{\fgl}(\bV)$ action $a$ on the irreducible Specht module $M=\bM_{\lambda}$. Since $\bV \otimes M$ (recall this is the induction product) is multiplicity-free by the Pieri rule, $a$ is simply multiplication by a scalar on each piece $\bM_\mu$. We claim that this scalar is the content of the box in the Young diagram $\mu \setminus \lambda$, where the content is its row index minus its column index, i.e., if $i$ is the unique index such that $\mu_i > \lambda_i$, the content is $\mu_i - i$.

To prove this, we can first use Schur--Weyl duality to translate this into a statement about Schur functors $\bS_\lambda$. The advantage is that we can evaluate on vector spaces of different dimensions to deduce the following:
\begin{enumerate}[\indent (1)]
\item The value of $a$ on $\bS_\mu(\bk^n)$ can be computed on a highest weight vector, so it is independent of $n$ as long as $n \ge \ell(\mu)$. So we may as well assume $n = \ell(\mu)$.

\item To compute $a$ on the tensor power $(\bk^n)^{\otimes d}$, we have
  \[
    a(e_i \otimes (e_{j_1} \otimes \cdots \otimes e_{j_d})) = \sum_{k=1}^d e_{j_k} \otimes (e_{j_1} \otimes \cdots \otimes e_i \otimes \cdots \otimes e_{j_d})
  \]
  where the sum is over all ways of swapping $e_i$ with some $e_{j_k}$. In particular, tensoring with the determinant character increases the eigenvalues of $a$ by $1$, so using this and (1), we may as well assume that we are adding a box to the first column of $\lambda$.

\item As can be seen with the tensor power $(\bk^n)^{\otimes d}$ in (2), applying the transpose duality to $a$ multiplies its eigenvalues by $-1$ since this affects Schur--Weyl duality by tensoring the usual $\fS_d$-action on tensor powers with the sign character. So to add a box to the first column, we just need to understand adding a box to the first row.

\item Iterating, we reduce to the case that $\lambda=\emptyset$ and $\mu=(1)$. In that case, it follows immediately that $a$ is the $0$ map, which is the content of the box that we added. \qedhere
\end{enumerate}
\end{remark}

\subsection{In $\OB$-modules}

Let $\OB$ be the category of finite totally ordered sets and order-preserving bijections, and let $\Mod_{\OB}$ denote the category of $\OB$-modules. Given $\OB$-modules $M$ and $N$, their \defi{shuffle tensor product} is
\begin{displaymath}
(M \otimes_{\rm shuff} N)(S) = \bigoplus_{S=A \amalg B} M(A) \otimes N(B),
\end{displaymath}
where the sum is over all partitions of $S$ into two disjoint sets $A$ and $B$, and $A$ and $B$ are given the induced order. The shuffle tensor product gives the category $\Mod_{\OB}$ of $\OB$-modules the structure of a symmetric monoidal category. (Note: $\OB$-modules are equivalent to graded vector spaces, but the shuffle tensor product does not correspond with the usual tensor product of graded vector spaces.)

Let $V$ be the $\OB$-module that is $\bk$ in degree~1 and~0 in other degrees. We make one comment on $\ul{\fgl}(V)$-modules. Recall that $\fA_n$ is the monoid generated by $n$ commuting idempotents $e_1, \ldots, e_n$. The symmetric group $\fS_n$ acts on $\fA_n$, and so we can form the semi-direct product $\fS_n \ltimes \fA_n$. Define $\fM_n$ to be the submonoid of $\fS_n \ltimes \fA_n$ generated by the elements $s_i e_i$ and $e_i s_i=s_i e_{i+1}$ for $1 \le i \le n-1$, where $s_i$ is the transposition of $\fS_n$ that swaps $i$ and $i+1$. Then one can show that giving a $\ul{\fgl}(V)$-module is equivalent to giving a sequence of representations of the monoids $\bN \times \fM_n$ (compare with Remark~\ref{rmk:gl-seq}). Details and other examples in $\Mod_{\OB}$ can be found in \cite{jun}.

\subsection{Braidings} \label{ss:braidedgl}

Suppose that $\cC$ is a (not necessarily braided) tensor category and $V$ is a braided object of $\cC$, that is, we are given an isomorphism $\beta \colon V \otimes V \to V \otimes V$ such that the endomorphisms $\id \otimes \beta$ and $\beta \otimes \id$ of $V^{\otimes 3}$ satisfy the braid relation. We can define $\ul{\fgl}(V)$ in this setting, as follows: a $\ul{\fgl}(V)$-module is an object $M$ equipped with a map $a \colon V \otimes M \to V \otimes M$ satisfying
\begin{displaymath}
\beta^{-1} a \beta a - a \beta a \beta^{-1} = a \beta - \beta a.
\end{displaymath}
Here we have written $\beta$ for $\beta \otimes \id$ and $a$ for $\id \otimes a$. The inverses are included on some factors so that $M=V$ with $a=\beta$ defines a $\ul{\fgl}(V)$-module (the standard representation).

\begin{remark} \label{rmk:VA}
Let $\VA$ be the category of finite dimensional vector spaces over the finite field $\bF$, and let $\VB$ be the subcategory where the morphisms are isomorphisms. We assume $\operatorname{char}(\bF)$ is invertible in $\bk$. Given two $\VB$-modules $M$ and $N$, we define their \defi{parabolic tensor product} by
\begin{displaymath}
(M \otimes_{\rm par} N)(X) = \bigoplus_{Y \subseteq X} M(Y) \otimes N(X/Y)
\end{displaymath}
where the sum is over all subspaces $Y \subseteq X$. This tensor product has a natural braiding, first considered by Joyal--Street \cite{joyal-street}. We expect that $\VA$-modules can be expressed as a curried structure in the braided category of $\VB$-modules. It would be interesting to understand $\ul{\fgl}(V)$-modules where $V$ is the standard $\VB$-module (i.e., $V(X)=\bk$ if $X$ is one-dimensional and $V(X)=0$ otherwise).
\end{remark}

\section{The symplectic Lie algebra} \label{s:symp}

\subsection{Currying}

Let $V$ be a finite-dimensional vector space. The space $V \oplus V^*$ carries a natural symplectic form, and so we can consider the corresponding symplectic Lie algebra $\fsp(V \oplus V^*)$. This algebra admits a decomposition
\begin{displaymath}
\fsp(V \oplus V^*) = \Div^2(V^*) \oplus \fgl(V) \oplus \Div^2(V).
\end{displaymath}
We thus see that giving a linear map
\begin{displaymath}
\mu \colon \fsp(V \oplus V^*) \otimes M \to M
\end{displaymath}
is equivalent to giving linear maps
\begin{displaymath}
a \colon V \otimes M \to V \otimes M, \qquad b \colon \Div^2(V) \otimes M \to M, \qquad b' \colon M \to \Sym^2(V) \otimes M.
\end{displaymath}

\begin{proposition} \label{prop:sp-curry}
Let $\mu$ and $(a,b,b')$ as above correspond. Then $\mu$ defines a representation of $\fsp(V \oplus V^*)$ if and only if $(a,b,b')$ satisfy the following conditions:
\begin{enumerate}
\item $a$ satisfies \eqref{eq:glid}, that is, it defines a $\ul{\fgl}(V)$ structure on $M$.
\item $b b_2=b b_1$ holds as maps $\Div^2 V \otimes \Div^2 V \otimes M \to M$ and $b'_1 b'=b'_2 b'$ holds as maps $M \to \Sym^2 V \otimes \Sym^2 V \otimes M$, that is, the multiplication defined by $b$ is commutative and the co-multiplication defined by $b'$ is co-commutative.
\item $b$ and $b'$ are maps of $\ul{\fgl}(V)$-modules, where $\Div^2(V)$ and $\Sym^2(V)$ are equipped with their natural $\ul{\fgl}(V)$ actions.
\item $b'b-b_1 b'_2 = (m \otimes 1)(1 \otimes a)(\Delta \otimes 1)$ holds as maps $\Div^2 V \otimes M \to \Sym^2 V \otimes M$. Here $\Delta$ is comultiplication and $m$ is multiplication.
\end{enumerate}
\end{proposition}

\begin{proof}
Let $M$ be a $\fsp(V\oplus V^*)$-module. (a) and (b) are translations of the relations satisfied by the subalgebras spanned by each component $\fgl(V)$, $\Div^2(V^*)$ and $\Div^2(V)$ while (c) is a translation of the relation between $\fgl(V)$ and the two components $\Div^2(V)$ and $\Div^2(V^*)$.

For (d), let $v_1,\dots,v_n$ be a basis for $V$. Pick $x \in M$ and $v_iv_j \in \Div^2 V$. Then
\begin{align*}
  b'(b(v_iv_j \otimes x)) = \sum_k v_k^2 \otimes (v_k^*)^{[2]}(v_iv_j) x + \sum_{k < \ell} v_kv_\ell \otimes (v^*_kv^*_\ell) (v_iv_j) x
\end{align*}
and
\begin{align*}
b_1(b'_2(v_iv_j \otimes x)) = \sum_k v_k^2 \otimes (v_iv_j)(v_k^*)^{[2]} x + \sum_{k<\ell} v_kv_\ell \otimes (v_iv_j)(v^*_kv^*_\ell) x.
\end{align*}
Next
\begin{displaymath}
[v_k^*v_\ell^*, v_iv_j] = \delta_{i,k} v^*_\ell v_j + \delta_{j,k} v^*_\ell v_i + \delta_{i,\ell} v_k^* v_j + \delta_{j,\ell} v_k^* v_i, 
\end{displaymath}
and so
\begin{displaymath}
(b'b-b_1b'_2)(v_iv_j \otimes x) =  \sum_k v_i v_k \otimes (v_k^* v_j) x  + \sum_k v_j v_k \otimes (v_k^* v_i) x.
\end{displaymath}
This is clearly the same as $(m \otimes 1)(1 \otimes a)(\Delta \otimes 1)$. If $2$ is invertible in $\bk$, then we are done. To finish the remaining case, we also need to show that these maps agree on the elements $v_i^{[2]} \otimes x$. The calculation is similar to what we have explained above, with the final result being
\[
  (b'b-b_1b'_2)(v_i^{[2]} \otimes x) =  \sum_k v_i v_k \otimes (v_k^* v_i) x = (m \otimes 1)(1 \otimes a)(\Delta \otimes 1)(v_i^{[2]} \otimes x).
  \]
Conversely, if $M$ is equipped with the three maps $a,b,b'$, then we see that the corresponding action of $\fsp(V\oplus V^*)$ on $M$ respects the Lie bracket.
\end{proof}

\begin{definition} \label{def:curried-sp}
Let $V$ be an object of a tensor category $\cC$. We define a module over the \defi{curried symplectic algebra} $\ul{\fsp}(V \oplus V^*)$ to be an object $M$ of $\cC$ equipped with maps
\begin{displaymath}
a \colon V \otimes M \to V \otimes M, \qquad b \colon \Div^2V \otimes M \to M, \qquad b' \colon M \to \Sym^2V \otimes M.
\end{displaymath}
satisfying \cref{prop:sp-curry}{a}--\cref{prop:sp-curry}{d}.
\end{definition}

\subsection{In species} \label{ss:sp-sp}

Let $M$ be an $\FB$-module equipped with maps
\begin{displaymath}
a \colon \bV \otimes M \to \bV \otimes M, \qquad b \colon \Div^2(\bV) \otimes M \to M, \qquad b' \colon M \to \Sym^2(\bV) \otimes M.
\end{displaymath}
Let $\alpha$ and $\omega$ be the simple $(1,1)$- and $(0,0)$-operations corresponding to $a$ as in \eqref{eq:gl-fb}. Let $\beta$ and $\beta'$ be the symmetric $(0,2)$- and $(2,0)$-operations corresponding to $b$ and $b'$. Thus we have
\begin{displaymath}
b(t^{\{i,j\}} \otimes x) = \beta^S_{i,j}(x), \qquad
b'(y) = \sum_{\{i,j\} \subset S} t^{\{i,j\}} \otimes (\beta')^S_{i,j}(y).
\end{displaymath}
for $x \in M(S \setminus \{i,j\})$ and $y \in M(S)$.

\begin{proposition} \label{prop:sp-gl}
The triple $(a,b,b')$ defines a representation of $\ul{\fsp}(\bV \oplus \bV^*)$ on $M$ if and only if the following conditions hold (for all finite sets $S$):
\begin{enumerate}
\item $\alpha$, $\omega$, $\beta$, and $\beta'$ pairwise commute (and each commutes with itself).
\item Given $i,j,k \in S$ distinct, we have $\alpha^{S \setminus i}_{j,k} \circ \alpha^{S \setminus k}_{i,j} = \alpha^{S \setminus j}_{i,k}$.
\item Given $i,j,k \in S$ distinct, we have $\alpha^S_{i,j} \circ \beta^{S \setminus j}_{i,k}=\beta^{S \setminus i}_{j,k}$, and similarly $(\beta')^{S \setminus i}_{j,k} \circ \alpha_{i,j}^S = (\beta')^{S \setminus j}_{i,k}$.
\item Given $i,j,k \in S$ distinct, we have $(\beta')^S_{i,j} \circ \beta^S_{j,k}=\alpha_{i,k}^{S \setminus j}$.
\item Given $i,j \in S$ distinct, we have $(\beta')^S_{i,j} \circ \beta^S_{i,j}=2\omega^{S \setminus \{i,j\}}$.
\end{enumerate}
\end{proposition}

\begin{proof}
Suppose that conditions \cref{prop:sp-gl}{a}--\cref{prop:sp-gl}{d} above hold. We verify that $(a,b,b')$ satisfy conditions \cref{prop:sp-curry}{a}--\cref{prop:sp-curry}{d}. Condition \cref{prop:sp-curry}{a} follows from \cref{prop:sp-gl}{a}, \cref{prop:sp-gl}{b}, and Proposition~\ref{prop:gl-fb}. Condition \cref{prop:sp-curry}{b} follows easily from \cref{prop:sp-gl}{a}.

We now verify \cref{prop:sp-curry}{c}. Let
\begin{displaymath}
a' \colon \bV \otimes \Div^2(\bV) \to \bV \otimes \Div^2(\bV)
\end{displaymath}
be the $\ul{\fgl}(\bV)$-action on $\Div^2(V)$. To show that $b$ is $\ul{\fgl}(V)$-equivariant, we must show that the diagram
\begin{displaymath}
\xymatrix@C=4em{
\bV \otimes \Div^2(\bV) \otimes M \ar[r]^-{1 \otimes b} \ar[d]_{a+a'} &
\bV \otimes M \ar[d]^a \\
\bV \otimes \Div^2(\bV) \otimes M \ar[r]^-{1 \otimes b} &
\bV \otimes M }
\end{displaymath}
commutes; recall from \S \ref{ss:glgen} that $a+a'$ defines the tensor product representation. Let $i,j,k \in S$ be distinct and let $x \in M(S \setminus \{i,j,k\})$. We have
\begin{displaymath}
a'(t^i \otimes t^{\{j,k\}})=t^j \otimes t^{\{i,k\}} + t^k \otimes t^{\{i,j\}},
\end{displaymath}
and so
\begin{align*}
(a+a')(t^i \otimes t^{\{j,k\}} \otimes x)
=& t^j \otimes t^{\{i,k\}} \otimes x + t^k \otimes t^{\{i,j\}} \otimes x + \\
&t^i \otimes t^{\{j,k\}} \otimes \omega^{S \setminus \{i,j,k\}}(x) + \sum_{\ell \in S \setminus \{i,j,k\}} t^{\ell} \otimes t^{\{j,k\}} \otimes \alpha_{\ell,i}^{S \setminus \{j,k\}}(x),
\end{align*}
and so
\begin{align*}
(1 \otimes b)(a+a')(t^i \otimes t^{\{j,k\}} \otimes x)
=& t^j \otimes \beta^{S \setminus j}_{i,k}(x) + t^k \otimes \beta^{S \setminus k}_{i,j}(x) \\
& t^i \otimes \beta^{S \setminus i}_{j,k}(\omega^{S \setminus \{i,j,k\}}(x)) +
\sum_{\ell \in S \setminus \{i,j,k\}} t^{\ell} \otimes \beta^{S \setminus \ell}_{j,k}(\alpha^{S \setminus \{j,k\}}_{\ell,i}(x)).
\end{align*}
On the other hand, we have
\begin{align*}
a(1 \otimes b)(t^i \otimes t^{\{j,k\}} \otimes x)
&= a(t^i \otimes \beta^{S \setminus i}_{j,k}(x)) \\
&= t^i \otimes \omega^{S \setminus i}(\beta^{S \setminus i}_{j,k}(x)) + \sum_{\ell \in S \setminus i} t^{\ell} \otimes \alpha^S_{\ell,i}(\beta^{S \setminus i}_{j,k}(x)).
\end{align*}
The above two expressions coincide if and only if the following equations hold (for $\ell \in S \setminus \{i,j,k\}$):
\begin{align*}
\beta^{S \setminus i}_{j,k}(\omega^{S \setminus \{i,j,k\}}(x)) &= \omega^{S \setminus i}(\beta^{S \setminus i}_{j,k}(x)) &
\beta^{S \setminus \ell}_{j,k}(\alpha^{S \setminus \{j,k\}}_{\ell,i}(x)) &= \alpha^S_{\ell,i}(\beta^{S \setminus i}_{j,k}(x)) \\
\beta^{S \setminus j}_{i,k}(x) &= \alpha^S_{j,i}(\beta^{S \setminus i}_{j,k}(x)) &
\beta^{S \setminus k}_{i,j}(x) &= \alpha^S_{k,i}(\beta^{S \setminus i}_{j,k}(x))
\end{align*}
The two equalities on the first line follow since $\beta$ commutes with $\alpha$ and $\omega$ by \cref{prop:sp-gl}{a}. The equalities on the second line are \cref{prop:sp-gl}{c}. This shows that $b$ is $\ul{\fgl}(\bV)$-equivariant. The proof for $b'$ is similar.

We now verify \cref{prop:sp-curry}{d}. We have
\begin{align*}
b'(b(t^{\{i_1,i_2\}} \otimes x)) &= \sum_{\{i_3, i_4\} \subset S} t^{\{i_3,i_4\}} \otimes \beta'_{i_3,i_4}(\beta_{i_1,i_2}(x)),\\
b_1b'_2(t^{\{i_1, i_2\}} \otimes x) &= \sum_{\{i_3, i_4\} \subset S \setminus \{i_1,i_2\}} t^{\{i_3, i_4\}} \otimes \beta_{i_1,i_2}(\beta'_{i_3,i_4}(x)).
\end{align*}
Now, for $\{i_3, i_4\} \subseteq S \setminus \{i_1,i_2\}$, we have $\beta_{i_1,i_2} \beta'_{i_3,i_4}=\beta'_{i_3,i_4} \beta_{i_1,i_2}$ since $\beta$ and $\beta'$ commute. We thus see that
\begin{displaymath}
(b'b-b_1b'_2)(t^{\{i_1, i_2\}} \otimes x) = \sum_{\substack{\{i_3,i_4\} \subseteq S\\ \{i_3,i_4\} \cap \{i_1,i_2\} \ne \emptyset}} t^{\{i_3, i_4\}} \otimes \beta'_{i_3,i_4}(\beta_{i_1,i_2}(x)).
\end{displaymath}
On the other hand,
\begin{align*}
& (m \otimes 1)(1 \otimes a)(\Delta \otimes 1)(t^{\{i_1,i_2\}} \otimes x) \\
=& 2 t^{\{i_1,i_2\}} \otimes \omega(x) + \sum_{j \in S \setminus \{i_1,i_2\}} (t^{\{j,i_2\}} \otimes \alpha_{j,i_1}(x) + t^{\{j,i_1\}} \otimes \alpha_{j,i_2}(x)).
\end{align*}
We claim these last two expressions coincide. The coefficient of $t^{\{i_1, i_2\}}$ in the first expression is $\beta'_{i_1,i_2}(\beta_{i_1,i_2}(x))$ and in the second expression is $2 \omega(x)$. These are equal by \cref{prop:sp-gl}{e}. Suppose now that $j \not\in \{i_1,i_2\}$. The $t^{\{i_1,j\}}$ component in the first expression is $\beta'_{i_1,j}(\beta_{i_1,i_2}(x))$ and in the second expression is $\alpha_{j,i_2}(x)$. These are equal by \cref{prop:sp-gl}{d}. The other components are similar.

This verifies the conditions \cref{prop:sp-curry}{a}--\cref{prop:sp-curry}{d}. This reasoning is completely reversible, and so the result follows.
\end{proof}

\subsection{The Brauer category} \label{ss:brauer}

Let $\fG=\fG(\delta)$ be the Brauer category with parameter $\delta \in \bk$. The objects of this category are finite sets. The space $\Hom_{\fG}(S,T)$ of morphisms is the vector space spanned by Brauer diagrams from $S$ to $T$; such a diagram is simply a perfect matching on the set $S \amalg T$. For the definition of composition (and additional details), see \cite[\S 5]{brauercat1}. We note that the composition law depends on the parameter $\delta$.

A Brauer diagram $S \to T$ is called upwards if there are no edges contained in $S$. The upwards Brauer category $\fU$ is the subcategory of $\fG$ containing all objects and where $\Hom_{\fU}(S,T)$ is spanned by upwards diagrams. There is a similarly defined downwards Brauer category $\fD$. The intersection $\fM$ of $\fU$ and $\fD$ is the linearization of $\FB$: that is, $\Hom_{\fM}(S,T)$ is the vector space spanned by bijections $S \to T$. The pair $(\fU, \fD)$ is a triangular structure on $\fG$, see \cite[Proposition~5.5]{brauercat1}. (Note that \cite{brauercat1} works in characteristic~0, but this statement and its proof hold in general.)

Suppose that $M$ is a $\fG$-module. Restricting to $\FB \subset \fG$, we can regard $M$ as an $\FB$-module. Let $S$ be a finite set and let $i,j \in S$ be distinct elements. We have a morphism $\eta^S_{i,j} \colon S \setminus \{i,j\} \to S$ in $\fG$ corresponding to the diagram with an edge between $i$ and $j$ in the target, and that is the identity elsewhere. This induces a linear map
\begin{displaymath}
\beta^S_{i,j} \colon M(S \setminus \{i,j\}) \to M(S).
\end{displaymath}
One easily sees that $\beta$ is a symmetric $(0,2)$-operation on $M$. Similarly, we have a morphism $(\eta')^S_{i,j} \colon S \to S \setminus \{i,j\}$ in $\fG$ using the opposite diagram, and this induces a linear map
\begin{displaymath}
(\beta')^S_{i,j} \colon M(S) \to M(S \setminus \{i,j\}).
\end{displaymath}
As above, $\beta'$ is a symmetric $(2,0)$-operation on $M$. Using the rule for composition in $\fG$, one easily sees that $\beta$ and $\beta'$ satisfy the following conditions (in what follows, $S$ is a finite set):
\begin{enumerate}
\item $\beta$ and $\beta'$ commute with themselves and with each other.
\item Let $i,j,k \in S$ be distinct. Then $(\beta')^S_{i,j} \beta^S_{j,k}=\iota^{S \setminus j}_{i,k}$.
\item Let $i,j \in S$ be distinct. Then $(\beta')^S_{i,j} \beta^S_{i,j}=\delta \cdot {\rm id}$.
\end{enumerate}
Since the $\eta$ and $\eta'$ morphisms, together with the morphisms in $\FB$, generate $\fG$, we see that that the operations $\beta$ and $\beta'$ completely determine the $\fG$-structure on $M$. The following proposition shows that the above conditions exactly characterize the operations we see in this manner:

\begin{proposition} \label{prop:brauer-op}
Let $M$ be an $\FB$-module equipped with a symmetric $(0,2)$-operation $\beta$ and a symmetric $(2,0)$-operation $\beta'$ satisfying (a), (b), and (c) above. Then $M$ carries a unique $\fG$-structure inducing $\beta$ and $\beta'$.
\end{proposition}

\begin{proof}
We claim that giving a $\fU$-structure on $M$ is equivalent to giving a self-commuting symmetric $(0,2)$-operation. First, suppose that $M$ has a $\fU$-structure. For any set $S$ and distinct elements $i,j \in S$, we define $\beta^S_{i,j}$ to be the action of $\eta^S_{i,j}$ on $M$. Then $\beta$ is a symmetric operation by construction and commutes with itself (if we compose such morphisms, the result does not depend on the order in which we pair off the elements).

Conversely, suppose that $M$ has a self-commuting symmetric $(0,2)$-operation $\beta$. We use $\beta$ to construct a $\fU$-structure on $M$. A $\fU$-morphism $\phi \colon S \to T$ can be factored as a bijection $\sigma \colon S \to \phi(S)$ followed by morphisms of the form $\eta^U_{i,j}$ where $i,j$ are distinct. We define $M_{\phi} \colon M(S) \to M(T)$ to be the composition $M_\sigma \colon M(S) \to M(\phi(S))$ with the corresponding composition of maps given by $\beta$ coming from the factorization; since $\beta$ is self-commuting the order of the factorization does not affect the result, and since it is symmetric the order of the elements $i,j$ at each stage also does not affect the result. Given another $\fU$-morphism $\psi \colon T \to U$, the functoriality $M_{\psi \phi} = M_{\psi} M_{\phi}$ follows from the naturality condition on operations (we omit the details).

Similarly, giving a $\fD$-structure on $M$ is equivalent to giving a self-commuting symmetric $(2,0)$-operation. Thus $\beta$ and $\beta'$ define $\fU$- and $\fD$-structures on $M$.

Let $\cU$ be the class of morphisms in $\fU$ isomorphic to $\eta^S_{i,j}$ for some $S$, $i$, and $j$, and define $\cD$ similarly using $\eta'$. One easily sees that $\cU$ generates $\fU$ and $\cD$ generates $\fD$. Thus, by Proposition~\ref{prop:tri-comp}, it suffices to show that $(\phi, \psi)$ is compatible for $\phi \in \cU$ and $\psi \in \cD$ with $\psi \circ \phi$ defined. Let $\phi = \eta^S_{i,j}$ and $\psi = (\eta')^S_{k,\ell}$. There are three cases to consider depending on the cardinality $n$ of $\{i,j\} \cap \{k,\ell\}$. 

First suppose $n=0$. Then
\begin{displaymath}
\psi \circ \phi = (\eta')^S_{k,\ell} \circ \eta^S_{i,j} = \eta^{S \setminus \{k,\ell\}}_{i,j} \circ (\eta')^{S \setminus \{i,j\}}_{k,\ell},
\end{displaymath}
where the second equality comes from the following composition of Brauer diagrams:
\[
\begin{tikzpicture}[baseline={([yshift=-.5ex]current bounding box.center)}]
\node at (2,2) (a3) {\tiny $i$};
\node at (3,2) (a4) {\tiny $j$};
\node [node] at (0,1) (b1) {};
\node [node] at (1,1) (b2) {};
\node [node] at (2,1) (b3) {};
\node [node] at (3,1) (b4) {};      
\node at (0,0) (c1) {\tiny $k$};
\node at (1,0) (c2) {\tiny $\ell$};
\draw[thick, orange] (c1) to (b1);
\draw[thick, orange] (c2) to (b2);
\draw[thick, orange] (b3) to (a3);
\draw[thick, orange] (b4) to (a4);
\draw[thick, orange] (b1) to[out=20, in=160] (b2);
\draw[thick, orange] (b3) to[out=-20, in=-160] (b4);      
\end{tikzpicture}
\quad = \quad
\begin{tikzpicture}[baseline={([yshift=-.5ex]current bounding box.center)}]
\node at (0,2) (a3) {\tiny $i$};
\node at (1,2) (a4) {\tiny $j$};
\node at (0,0) (c1) {\tiny $k$};
\node at (1,0) (c2) {\tiny $\ell$};
\draw[thick, orange] (c1) to[out=20, in=160] (c2);
\draw[thick, orange] (a3) to[out=-20, in=-160] (a4);      
\end{tikzpicture}.
\]
We thus have
\begin{displaymath}
(\psi \circ \phi)_*=\beta^{S \setminus \{k,\ell\}}_{i,j} \circ (\beta')^{S \setminus \{i,j\}}_{k,\ell}=(\beta')^S_{k,\ell} \circ \beta^S_{i,j} = \psi_* \circ \phi_*,
\end{displaymath}
where the first equality uses the above computation and the second uses condition~(a). Thus $(\phi,\psi)$ is compatible.

Now suppose $n=1$, and, without loss of generality, $j=k$. Then
\begin{displaymath}
\psi \circ \phi = (\eta')^S_{j,\ell} \circ \eta^S_{i,j} = \iota_{i,\ell}^{S \setminus j}
\end{displaymath}
where the second equality comes from the following composition of Brauer diagrams:
\[
\begin{tikzpicture}[baseline={([yshift=-.5ex]current bounding box.center)}]
\node at (2,2) (a3) {\tiny $i$};
\node [node] at (0,1) (b1) {};
\node [node] at (1,1) (b2) {};
\node [node] at (2,1) (b3) {};
\node at (0,0) (c1) {\tiny $\ell$};
\draw[thick, orange] (c1) to (b1);
\draw[thick, orange] (b3) to (a3);
\draw[thick, orange] (b1) to[out=20, in=160] (b2);
\draw[thick, orange] (b2) to[out=-20, in=-160] (b3);
\end{tikzpicture}
\quad = \quad
\begin{tikzpicture}[baseline={([yshift=-.5ex]current bounding box.center)}]
\node at (0,2) (a3) {\tiny $i$};
\node at (0,0) (c1) {\tiny $\ell$};
\draw[thick, orange] (c1) to (a3);
\end{tikzpicture}.
\]
We thus have
\begin{displaymath}
(\psi \circ \phi)_* = \iota_{i,\ell}^{S \setminus j} = (\beta')^S_{j,\ell} \circ \beta^S_{i,j} = \psi_* \circ \phi_*
\end{displaymath}
by (b), which establishes the compatibility.

Finally, suppose $n=2$. Then
\begin{displaymath}
\psi \circ \phi = (\eta')^S_{i,j} \circ \eta^S_{j,i} = \delta \cdot {\rm id}
\end{displaymath}
using the composition
\[
\begin{tikzpicture}[baseline={([yshift=-.5ex]current bounding box.center)}]
\node [node] at (0,0) (a1) {};
\node [node] at (1,0) (a2) {};
\draw[thick, orange] (a1) to[out=20, in=160] (a2);
\draw[thick, orange] (a1) to[out=-20, in=-160] (a2);
\end{tikzpicture}
\quad = \quad \delta, 
\]
and so
\begin{displaymath}
(\psi \circ \phi)_* = \delta = (\beta')^S_{i,j} \circ \beta^S_{j,i} = \psi_* \circ \phi_*
\end{displaymath}
by (c), which establishes the compatibility.
\end{proof}

\subsection{The comparison theorem}

Fix $\delta \in \bk$. If $M$ is a representation of $\ul{\fsp}(\bV \oplus \bV^*)$ given by data $(a,b,b')$ then $a$ defines a representation of $\ul{\fgl}(\bV)$ on $M$. We say that $M$ is \defi{$\delta$-standard} if the representation of $\ul{\fgl}(\bV)$ is $\delta$-standard (see \S \ref{ss:gl-fb}). We let $\Rep_{\delta}(\ul{\fsp}(\bV \oplus \bV^*))$ be the category of $\delta$-standard representations.

We define a functor
\[
  \Phi \colon \Mod_{\fG(2\delta)} \to \Rep_\delta(\ul{\fsp}( \bV \oplus \bV^*) )
\]
as follows. Let $M$ be a representation of $\fG(2\delta)$. To define $\Phi(M)$, we only need to define the operations $\alpha$, $\omega$, $\beta$ and $\beta'$. First, $M$ is an $\FB$-module by restriction, and we choose $\alpha$ and $\omega$ as in \S\ref{ss:gl-fb} so that the result is $\delta$-standard. The operations $\beta$ and $\beta'$ are defined using the morphisms $\eta$ and $\eta'$ as in \S\ref{ss:brauer}. Then $\Phi(M)$ is indeed an object of $\Rep_\delta(\ul{\fsp}(\bV \oplus \bV^*))$ by Proposition~\ref{prop:sp-gl}. For a morphism $f \colon M \to N$, we let $\Phi(f)$ be the same morphism of the underlying $\FB$-modules.

The following is the main result of this section:

\begin{theorem} \label{thm:fB=B}
The functor $\Phi$ defines a natural isomorphism of categories
\begin{displaymath}
\Mod_{\fG(2\delta)} \cong \Rep_{\delta}(\ul{\fsp}(\bV \oplus \bV^*)).
\end{displaymath}
\end{theorem}

\begin{proof}
The inverse of $\Phi$ is defined by reversing the steps in the definition of $\Phi$. This is well-defined by Proposition~\ref{prop:brauer-op}.
\end{proof}

\begin{remark}
  If 2 is a zerodivisor in $\bk$, then $\delta$ cannot generally be recovered from $2\delta$. In particular, if $2\delta=2\delta'$, we see that there is an equivalence of the form
  \[
    \Rep_\delta(\ul{\fsp}(\bV \oplus \bV^*)) \cong    \Rep_{\delta'}(\ul{\fsp}(\bV \oplus \bV^*)). \qedhere
  \]
\end{remark}

\section{The Witt algebra} \label{s:witt}

\subsection{Currying} \label{ss:witt-curry}

Let $V$ be a finite-dimensional vector space with basis $\{\xi_i\}$. The \defi{Witt algebra} on $V$, denote $W(V)$, is the Lie algebra of $\bk$-linear derivations of the polynomial ring $\bk[\xi_i]$. Thus it is spanned by elements $f \partial_i$ where $f$ is a polynomial in $\{\xi_j\}$ and $\partial_i$ is the partial derivative with respect to $\xi_i$, and the bracket is given by
\[
  [f \partial_i, g\partial_j] = f \frac{\partial g}{\partial \xi_i} \partial_j - g \frac{\partial f}{\partial \xi_j} \partial_i.
\]
We have a canonical isomorphism of vector spaces $W(V) = \bigoplus_{n \ge 0} \Sym^{n}(V) \otimes V^*$.

\begin{remark}
The algebra of derivations of the ring $\bk[z,z^{-1}]$ of Laurent polynomials is also sometimes referred to as the Witt algebra. We do not know of a good analogue of this Lie algebra in the multivariate case.
\end{remark}

A linear map $\mu \colon W(V) \otimes M \to M$ is the same data as linear maps
\[
  a^{(n)} \colon \Sym^{n+1} V \otimes M \to V \otimes M, \qquad n \ge -1.
\]
For notational simplicity, we package these together for all $n \ge -1$ into a single map
\[
  a \colon \Sym V \otimes M \to V \otimes M.
\]
Fix a map $a$ as above. We define a map
\[
  a' \colon \Sym V \otimes \Sym V \otimes M \to V \otimes V \otimes M
\]
as the composition
\begin{align*}
  \Sym V \otimes \Sym V \otimes M &\xrightarrow{{\rm id} \otimes \Delta \otimes {\rm id}} \Sym V \otimes V \otimes \Sym V \otimes M\\
  &\xrightarrow{\tau \otimes {\rm id} \otimes {\rm id}}  V \otimes \Sym V \otimes \Sym V \otimes M\\
                                  &\xrightarrow{{\rm id} \otimes m \otimes {\rm id}} V \otimes \Sym V \otimes M\\
  &\xrightarrow{{\rm id} \otimes a} V \otimes V \otimes M,
\end{align*}
where $\Delta \colon \Sym V \to V \otimes \Sym V$ is the comultiplication given by $f \mapsto \sum_{i=1}^n x_i \otimes \frac{\partial f}{\partial x_i}$, $m \colon \Sym V \otimes \Sym V \to \Sym V$ is the multiplication map, and $\tau$ is the usual switching map. We define $a'' = \tau a' \tau$.

\begin{proposition} \label{prop:witt-curry}
  Let $\mu$ and $a$ be corresponding linear maps as above. Then $\mu$ defines a representation of $W$ if and only if $[a_1,a_2]=a'-a''$ holds as maps $\Sym V \otimes \Sym V \otimes M \to V \otimes V \otimes M$.
\end{proposition}

\begin{proof}
  Pick $\xi^\alpha,\xi^\beta \in \Sym V$ and $x \in M$. Let $\eps_i$ be the exponent vector which is 1 in position $i$ and 0 elsewhere. Then 
  \[
    [a_1,a_2](\xi^\alpha \otimes \xi^\beta \otimes x) = \sum_{i,j} \xi_j \otimes \xi_i \otimes [\xi^\alpha \partial_j, \xi^\beta \partial_i] x,
  \]
  and
  \[
    [\xi^\alpha \partial_j, \xi^\beta \partial_i] = \beta_j \xi^{\alpha+\beta-\eps_j} \partial_i - \alpha_i \xi^{\alpha+\beta-\eps_i} \partial_j.
  \]
  Next, we compute the effect of $a'$ via the maps it is a composition of:
  \begin{align*}
    \xi^\alpha \otimes \xi^\beta \otimes x &\mapsto \sum_j \beta_j \xi^\alpha \otimes \xi_j \otimes \xi^{\beta-\eps_j} \otimes x\\
                                       &\mapsto \sum_j \beta_j \xi_j \otimes  \xi^{\alpha+\beta-\eps_j} \otimes x\\
    &\mapsto \sum_{i,j} \beta_j \xi_j \otimes \xi_i \otimes (\xi^{\alpha+\beta-\eps_j} \partial_i) x.
  \end{align*}
  So
  \[
    a''(\xi^\alpha \otimes \xi^\beta \otimes x) = \sum_{i,j} \alpha_j \xi_i \otimes \xi_j \otimes (\xi^{\alpha + \beta - \eps_j} \partial_i) x.
    \]
    Hence we see that $[a_1,a_2]=a'-a''$. On the other hand, if $[a_1,a_2]=a'-a''$, then the above calculations show that
    \[
      [\xi^\alpha \partial_j, \xi^\beta \partial_i] x = (\xi^{\alpha} \partial_j)(\xi^\beta \partial_i)x - (\xi^\beta \partial_i) (\xi^{\alpha} \partial_j) x
    \]
    which shows that $\mu$ defines a Lie algebra action on $M$.
\end{proof}

\begin{definition}
Given an object $V$ of a tensor category $\cC$, we define a module over the \defi{curried Witt algebra} $\ul{W}(V)$ to be an object $M$ together with a map $a \colon \Sym V \otimes M \to V \otimes M$ such that $[a_1,a_2] = a'-a''$ with notation as in Proposition~\ref{prop:witt-curry}.
\end{definition}

\begin{remark}
There are several variants of the above definition one can consider. For instance, one can consider the Lie subalgebra $W^+(V)$ of $W(V)$ consisting of derivations $f \partial_i$ where $f$ has no constant term; it has a curried form $\ul{W}^+(V)$ similar to that for $W(V)$. One can also define a curried algebra $\ul{W}(V^*)$ by considering maps $V \otimes M \to \Div(V) \otimes M$.
\end{remark}

\begin{proposition} \label{prop:witt-gl}
Let $(M,a)$ be a representation of $\ul{W}(V)$. Then $a^{(0)}$ defines a representation of $\ul{\fgl}(V)$ on $M$.
\end{proposition}

\begin{proof}
The restrictions of $a'$ and $a''$ to $V\otimes V \otimes M$, respectively, are $\tau a^{(0)}_1$ and $\tau a^{(0)}_2$, so the identity $[a^{(0)}_1,a^{(0)}_2] = \tau( a^{(0)}_1 - a^{(0)}_2 )$ follows immediately from the definition of a representation of $\ul{W}(V)$.
\end{proof}

\subsection{In species}

Let $M$ be an $\FB$-module equipped with a map
\begin{displaymath}
a \colon \Sym(\bV) \otimes M \to \bV \otimes M.
\end{displaymath}
Let $\phi$ be the symmetric $(1,\ast)$-operation associated to $a$; by $(1,\ast)$ we mean that $\phi_{A,B}=0$ unless $A$ has cardinality~1. Explicitly, for a finite set $S$, a subset $B$ of $S$, and $x \in M(S \setminus B)$, we have
\begin{displaymath}
a(t^B \otimes x) = \sum_{i \in S} t^i \otimes \phi^S_{i,B}(x).
\end{displaymath}
Let $\alpha=\phi[0]$ and $\omega=\phi[1]$ be the simple operations associated to $\phi$. Explicitly, for $S$, $B$, and $x$ as above, we have
\begin{displaymath}
a(t^B \otimes x) = \sum_{i \in B} t^i \otimes \omega^{S \setminus i}_{B \setminus i}(x) + \sum_{i \in S \setminus B} t^i \otimes \alpha^S_{i,B}(x).
\end{displaymath}
In general, $\omega^S_B$ is a map $M(S \setminus B) \to M(S)$.

\begin{proposition} \label{prop:witt-fb}
With notation as above, $a$ defines a representation of $\ul{W}(\bV)$ if and only if the following conditions hold ($S$ is a finite set and $A$ and $B$ are disjoint subsets of $S$):
\begin{enumerate}
\item The operations $\alpha$ and $\omega$ commute with themselves and each other.
\item Let $j \in B$ and $i \in S \setminus (A \cup B)$. Then $\alpha^{S \setminus i}_{j,A} \circ \alpha^{S \setminus A}_{i,B}=\alpha^{S \setminus j}_{i, A \cup B \setminus j}$.
\item Let $j \in B$. Then $\alpha^S_{j,A} \circ \omega^{S \setminus A}_B=\omega^{S \setminus j}_{A \cup B \setminus j}$.
\end{enumerate}
\end{proposition}

\begin{proof}
Let $A$ and $B$ be disjoint subsets of $S$ and let $x \in M(S \setminus (A \cup B))$. Then we have
\begin{align*}
a_1a_2(t^A \otimes t^B \otimes x)
&= \sum_{i \in S \setminus A} \sum_{j \in S \setminus i} t^j \otimes t^i \otimes \phi^{S \setminus i}_{j, A}(\phi^{S \setminus A}_{i, B}(x)) \\
a_2a_1(t^A \otimes t^B \otimes x)
&= \sum_{j \in S \setminus B} \sum_{i \in S \setminus j} t^j \otimes t^i \otimes \phi^{S \setminus j}_{i, B}(\phi^{S \setminus B}_{j, A}(x)).
\end{align*}
Next, we compute $a'$ as a composition of maps:
\begin{align*}
t^A \otimes t^B \otimes x
&\mapsto \sum_{j \in B} t^A \otimes t^j \otimes t^{B \setminus j} \otimes x\\
&\mapsto \sum_{j \in B} t^j \otimes t^{A \cup B \setminus j} \otimes x\\
&\mapsto \sum_{j \in B} \sum_{i \in S \setminus j} t^j \otimes t^i \otimes \phi^{S \setminus j}_{i, A \cup B \setminus j}(x).
\end{align*}
Similarly,
\begin{align*}
a''(t^A \otimes t^B \otimes x)
&= \sum_{i \in A} \sum_{j \in A \setminus i} t^j \otimes t^i \otimes \phi^{S \setminus i}_{j, A \cup B \setminus i}(x).
\end{align*}
Equating coefficients in the equation $[a_1,a_2]=a'-a''$ we find the following (for distinct $i, j\in S$):
\begin{enumerate}[(i)]
\item If $i \not\in A$ and $j \not\in B$ then $\phi^{S \setminus i}_{j,A} \circ \phi^{S \setminus A}_{i,B}=\phi^{S \setminus j}_{i,B} \circ \phi^{S \setminus B}_{j,A}$.
\item If $i \not\in A$ and $j \in B$ then $\phi^{S \setminus i}_{j,A}\circ \phi^{S \setminus A}_{i,B}=\phi^{S \setminus j}_{i,A \cup B \setminus j}$.
\item If $i \in A$ and $j \not\in B$ then $\phi^{S \setminus j}_{i,B}\circ \phi^{S \setminus B}_{j,A}=\phi^{S \setminus i}_{j, A \cup B \setminus i}$.
\item If $i \in A$ and $j \in B$ then $\phi^{S \setminus j}_{i,A \cup B \setminus j}=\phi^{S \setminus i}_{j, A \cup B \setminus i}$.
\end{enumerate}
Statement (i) is equivalent to (a); statement (ii) is equivalent to the conjunction of (b) and (c); statement (iii) is equivalent to statement (ii); and statement (iv) is automatic. The result follows.
\end{proof}

\subsection{The restricted partition category} \label{ss:res-part}

Let $\fG=\fG(\delta)$ be the partition category with parameter $\delta$. The objects of this category are finite sets. The space $\Hom_{\fG}(S,T)$ of morphisms is the vector space spanned by partition diagrams from $S$ to $T$; such a diagram is a set-partition of $S \amalg T$. For the definition of composition (and additional details), see \cite[\S 6]{brauercat1}. We note that the composition law depends on the parameter $\delta$.

A partition diagram $S \to T$ is called upwards if each part contains at least one element of $T$ and at most one element of $S$. The upwards partition category $\fU$ is the subcategory of $\fG$ containing all objects and where $\Hom_{\fU}(S,T)$ is spanned by upwards diagrams. There is a similarly defined downwards partition category $\fD$. The intersection $\fM$ of $\fU$ and $\fD$ is the linearization of $\FB$. The pair $(\fU, \fD)$ is a triangular structure on $\fG$, see \cite[Proposition~6.3]{brauercat1}. (Once again, note that \cite{brauercat1} works in characteristic~0, but this statement and its proof hold in general.)

We say that a partition diagram from $S$ to $T$ is \defi{restricted} if each part contains at most one element of $S$. We define the \defi{restricted partition category} $\fG^r=\fG^r(\delta)$ to be the subcategory of $\fG$ with all objects and where the $\Hom$ spaces are spanned by restricted partition diagrams. One readily verifies that this is indeed a subcategory of $\fG$. We let $\fU^r$ and $\fD^r$ be the intersections of $\fU$ and $\fD$ with $\fG^r$. One easily verifies that $(\fU^r, \fD^r)$ is a triangular structure on $\fG^r$.

Suppose that $M$ is a $\fG$-module. Restricting to $\FB \subset \fG$, we can regard $M$ as an $\FB$-module. Let $S$ be a finite set, let $A$ be a subset of $S$, and let $i \in S \setminus A$. We have a morphism $\eta^S_{i,A} \colon S \setminus A \to S \setminus i$ in $\fG$ corresponding to the diagram in which $A \cup \{i\}$ forms a single part, and the remaining diagram is the identity. This induces a linear map
\begin{displaymath}
\alpha^S_{i,A} \colon M(S \setminus A) \to M(S \setminus i).
\end{displaymath}
One easily sees that $\alpha$ is a simple symmetric $(1,\ast)$-operation on $M$. Similarly, we have a morphism $\zeta^S_A \colon S \setminus A \to S$ in $\fG$ in which $A$ forms a single part and the remaining diagram is the identity, and this induces a linear map
\begin{displaymath}
\omega^S_A \colon M(S \setminus A) \to M(S).
\end{displaymath}
Again, one verifies that $\omega$ is a symmetric $(0,\ast)$-operation on $M$. Using the rule for composition in $\fG$, one sees that $\sigma$ and $\omega$ satisfy conditions (a), (b), and (c) from Proposition~\ref{prop:witt-fb}, as well as the following:
\begin{enumerate} \setcounter{enumi}{3}
\item Let $i,j \in S$ be distinct, and put $A=\{j\}$. Then $\alpha^S_{i,A}=(\iota^S_{i,j})_*$. 
\item We have $\omega^S_{\emptyset}=\delta$.
\end{enumerate}
Since the $\eta$ and $\zeta$ morphisms generate $\fG$, we see that that the operations $\alpha$ and $\omega$ completely determine the $\fG$-structure on $M$. The following proposition shows that the above conditions exactly characterize the operations we see in this manner:

\begin{proposition} \label{prop:res-part-op}
Let $M$ be an $\FB$-module equipped with a simple symmetric $(1,\ast)$-operation $\alpha$ and a symmetric $(0,\ast)$-operation $\omega$ satisfying (a), (b), and (c) from Proposition~\ref{prop:witt-fb} and (d) and (e) above. Then $M$ carries a unique $\fG^r$-structure inducing $\alpha$ and $\omega$.
\end{proposition}

\begin{proof}
  Suppose $M$ is given with $\alpha$ and $\omega$ as in the statement of the proposition. We first show how to construct a $\fU^r$-structure on $M$. Let $\phi \colon S \to T$ be a $\fU^r$-morphism corresponding to a partition diagram. Let $B_1,\dots,B_r$ be the blocks of this diagram such that $|B_i \cap S|=1$ and let $B_1',\dots,B_s'$ be the blocks of this diagram such that $|B_i'\cap S|=0$. Let $X_i = B_i \cap T$ and $x_i$ be the unique element of $B_i \cap S$, and let $Y_i = B_i'$ (thought of as a subset of $T$). Also set $T' \setminus (Y_1 \cup \cdots \cup Y_s)$. We have the factorization
  \[
    \phi = \eta^{(T'\setminus (X_r \cup \cdots \cup X_2))\amalg \{x_1\} }_{x_1, X_1} \cdots \eta^{T' \amalg \{x_r\}}_{x_r, X_r} \zeta^{T \setminus (Y_s \cup \cdots \cup Y_2)}_{Y_1} \cdots \zeta^{T \setminus Y_s}_{Y_{s-1}} \zeta^T_{Y_s}
  \]
We define $M_{\phi} \colon M(S) \to M(T)$ by replacing each $\zeta$ above by $\omega$ and each $\eta$ by $\alpha$. By (a), $\alpha$ and $\omega$ self-commute so the order of the blocks does not affect the definition of $M_{\phi}$. Furthermore, since $\alpha$ and $\omega$ commute with each other, we could have alternatively factored $\phi$ as a product of $\eta$ and $\zeta$ in any order. This fact, together with conditions (b) and (c), say that for any other restricted upwards partition diagram $\psi \colon T\to U$, we have $M_{\psi} M_{\phi} = M_{\psi \phi}$.

We can do the same to give a $\fD^r$-structure on $M$. This is like the above case but we only use $\eta$ such that the $X_i$ have size at most 1. Condition (d) tells us that the restriction of the $\fD^r$ and $\fU^r$ structures to $\FB$ agree with the usual $\FB$-action, so in particular they agree with each other.

Let $\cU$ be the class of morphisms in $\fU$ isomorphic to $\eta^S_{i, A}$ for some $S$, $i$, and $A$ (with $|A|>1$), or $\alpha^S_A$ for some $S$ and $A$, and define $\cD$ similarly using $\eta^S_{i, \emptyset}$. One easily sees that $\cU$ generates $\fU$ and $\cD$ generates $\fD$. Thus, by Proposition~\ref{prop:tri-comp}, it suffices to show that $(\phi, \psi)$ is compatible for $\phi \in \cU$ and $\psi \in \cD$ with $\psi \circ \phi$ defined.

Let $\psi = \eta^S_{j,\emptyset}$. First suppose $\phi = \eta^{S\amalg \{i\}}_{i, A}$ for $i \notin S$. If $j \in A$, then compatibility follows from (b), and if $j \notin A$, then compatibility follows from (a) since $\alpha$ self-commutes. Now suppose that $\phi = \zeta^S_A$. If $j \in A$, then compatibility follows from (c) and if $j \notin A$, then compatibility follows again from (a) since $\alpha$ and $\omega$ commute with each other. 
\end{proof}

\subsection{The comparison theorem}

Recall from Proposition~\ref{prop:witt-gl} that if $M$ is a $\ul{W}(\bV)$-module then, restricting the action map to $\bV \subset \Sym(\bV)$, we obtain a representation of $\ul{\fgl}(\bV)$ on $M$. We say that $M$ is \defi{$\delta$-standard} if this representation of $\ul{\fgl}(\bV)$ is $\delta$-standard. We write $\Rep_{\delta}(\ul{W}(\bV))$ for the full subcategory of $\Rep(\ul{W}(\bV))$ spanned by the $\delta$-standard representations. The following is the main result of this section:

\begin{theorem} \label{thm:witt}
We have a natural isomorphism of categories:
\begin{displaymath}
\Mod_{\fG^r(\delta)} \cong \Rep_{\delta}(\ul{W}(\bV)).
\end{displaymath}
\end{theorem}

\begin{proof}
This follows from combining Propositions~\ref{prop:witt-fb} and~\ref{prop:res-part-op}. We note that conditions (d) and (e) in the latter correspond to the $\delta$-standard condition. 
\end{proof}

\subsection{Application to $\FA$} \label{ss:fa}

We now consider $\fG^r=\fG^r(0)$ with parameter $\delta=0$. Let $\FA$ be the category of finite sets and all functions. A function $f \colon T \to S$ can be viewed as a restricted partition diagram from $S$ to $T$: the parts are $\{x\} \cup f^{-1}(x)$ with $x \in S$. Furthermore, this identification is compatible with composition. We thus see that the linearized category $\bk[\FA^{\op}]$ is equivalent to the subcategory of $\fG^r$ spanned by the $\eta$ morphisms. Since $\delta=0$, the $\zeta$ morphisms form an ideal of $\fG^r$, and we have a $\bk$-linear functor $\fG^r \to \bk[\FA^{\op}]$ that kills the $\zeta$ morphisms. We therefore see that an $\FA^{\op}$-module is the same as a $\fG^r$-module in which the $\zeta$ morphisms act by zero. Let $\Rep_0'(\ul{W}(\bV))$ be the full subcategory of $\Rep(\ul{W}(\bV))$ spanned by representations that are 0-standard and in which the $\omega$ operation vanishes. Since the $\zeta$ morphisms correspond to the $\omega$ operation, Theorem~\ref{thm:witt} immediately implies the following:

\begin{proposition}
There is a natural isomorphism of categories
\begin{displaymath}
\Mod_{\FA^{\op}} = \Rep_0'(\ul{W}(\bV)).
\end{displaymath}
\end{proposition}

\begin{remark}
Let $\FS$ be the category of finite sets and surjections. There is an analog of the above proposition in which $\FA^{\op}$ is replaced with $\FS^{\op}$ and $\ul{W}(\bV)$ is replaced with $\ul{W}^+(\bV)$. This point of view will be pursued further in \cite{witt}.
\end{remark}

\section{The Weyl Lie algebra} \label{s:weyl}

\subsection{Currying}

Let $(U, \omega)$ be a finite-dimensional symplectic space. Recall that the \defi{Weyl algebra} $A=A(U)$ is the quotient of the tensor algebra $T(U)$ by the 2-sided ideal generated by the relations $xy-yx=\omega(x,y)$ with $x,y \in U$. We define the \defi{Weyl Lie algebra}, denoted $\fa=\fa(U)$, to be the Lie algebra of this associative algebra. Thus $\fa=A$ as a vector space, and the bracket in $\fa$ is the commutator bracket in $A$. Note that $\fa$-modules are vastly more complicated than $A$-modules; indeed an $\fa$-module is the same as a module over the universal enveloping algebra $\cU(\fa)$, which is much larger than $A$.

Let $V$ be a finite-dimensional vector space, and let $\fa=\fa(V \oplus V^*)$, where we regard $U=V \oplus V^*$ as a symplectic space in the usual manner. As a vector space, we have $\fa=\Sym(V \oplus V^*)$. We thus see that, for a vector space $M$, giving a map
\begin{displaymath}
\mu \colon \fa \otimes M \to M
\end{displaymath}
is equivalent to giving maps
\begin{displaymath}
\mu_{r,s} \colon \Sym^r(V) \otimes \Sym^s(V^*) \otimes M \to M
\end{displaymath}
for all $r,s \in \bN$, which, in turn, is equivalent to giving maps
\begin{displaymath}
a_{r,s} \colon \Sym^r(V) \otimes M \to \Div^s(V) \otimes M
\end{displaymath}
for all $r,s \in \bN$. We assume for simplicity, that for $f \in \Sym^r(V)$ and $x \in M$ we have $a_{r,s}(f \otimes x)=0$ for all but finitely many $s$; this will automatically hold in the main case of interest to us. We can therefore package the $a_{r,s}$'s into a single map
\begin{displaymath}
a \colon \Sym(V) \otimes M \to \Div(V) \otimes M.
\end{displaymath}
Given a map $a$ as above, we define $a'$ to be the composition in the following diagram (set $S=\Sym(V)$ and $D=\Div(V)$):
\begin{displaymath}
\resizebox{\textwidth}{!}{
\xymatrix@C=6em{
S \otimes S \otimes M \ar[r]^-{\id \otimes \Delta \otimes \id} \ar@{..>}[d]_{a'} &
S \otimes S \otimes S \otimes M \ar[r]^-{\tau \otimes \id \otimes \id} &
S \otimes S \otimes S \otimes M \ar[r]^-{\id \otimes m \otimes \id} &
S \otimes S \otimes M \ar[d]^{\id \otimes a} \\
D \otimes D \otimes M & 
D \otimes D \otimes D \otimes M \ar[l]_-{m \otimes \id \otimes \id} &
S \otimes D \otimes D \otimes M \ar[l]_-{\avg \otimes \id \otimes \id \otimes \id} &
S \otimes D \otimes M \ar[l]_-{\id \otimes \Delta \otimes \id} } }
\end{displaymath}
Here $m$ and $\Delta$ are multiplication and comultiplication, $\tau$ is the symmetry of the tensor product, and $\avg \colon S \to D$ is the averaging map. We define $a'' = \tau_{1,2} a' \tau_{1,2}$.

\begin{proposition} \label{prop:part-curry}
Let $\mu$ and $a$ be corresponding linear maps as above. Then $\mu$ defines a representation of $\fa$ if and only if $[a_1,a_2]=a'-a''$.
\end{proposition}

\begin{proof}
  Let $\xi_1, \ldots, \xi_n$ be a basis for $V$ and let $\eta_1,\dots,\eta_n$ be the dual basis for $V^*$. We identify $\Sym(V)$ with the polynomial ring $\bk[\xi_1, \ldots, \xi_n]$, and $\Sym(V^*)$ with the polynomial ring $\bk[\eta_1, \ldots, \eta_n]$. For an exponent vector $\alpha \in \bN^n$, we let $\xi^{\alpha}$ be the monomial $\xi_1^{\alpha_1} \cdots \xi_n^{\alpha_n}$. We also define the divided power $\xi^{[\alpha]}=\frac{\xi^{\alpha}}{\alpha!}$, where $\alpha!=(\alpha_1!) \cdots (\alpha_n!)$. We define $\eta^{\alpha}$ and $\eta^{[\alpha]}$ similarly. For $1 \le i \le n$, we let $\delta_i \in \bN^n$ be the exponent vector that is~1 in the $i$th coordinate and~0 elsewhere.

The identity map $\Sym^r(V^*) \to \Sym^r(V^*)$ curries to the map $\bk \to \Div^r(V) \otimes \Sym^r(V^*)$ taking~1 to $\sum_{\vert \sigma \vert=r} \xi^{[\sigma]} \otimes \eta^{\sigma}$. It follows that we have
\begin{displaymath}
a(\xi^{\alpha} \otimes x) = \sum_{\sigma} \xi^{[\sigma]} \otimes \xi^{\alpha} \eta^{\sigma} x,
\end{displaymath}
where the sum is over all exponent vectors, and $\xi^{\alpha} \eta^{\sigma}$ is regarded as an element of $\fa$. As usual, we thus have
\begin{displaymath}
[a_1,a_2](\xi^{\alpha} \otimes \xi^{\beta} \otimes x) = \sum_{\sigma,\tau} \xi^{[\sigma]} \otimes \xi^{[\tau]} \otimes [\xi^{\alpha} \eta^{\sigma}, \xi^{\beta} \eta^{\tau}] x.
\end{displaymath}
Now, in the Weyl algebra $A$ we have
\begin{displaymath}
\eta_i^r \xi_i^s = \sum_{\epsilon_i \in \bN} \binom{r}{\epsilon_i} \binom{s}{\epsilon_i} \epsilon_i! \cdot \xi_i^{s-\epsilon_i} \eta_i^{r-\epsilon_i},
\end{displaymath}
and so
\begin{displaymath}
[\xi^{\alpha} \eta^{\sigma}, \xi^{\beta} \eta^{\tau}] = \sum_{\epsilon \in \bN^n} \left( \binom{\beta}{\epsilon} \binom{\sigma}{\epsilon} - \binom{\alpha}{\epsilon} \binom{\tau}{\epsilon} \right) \epsilon! \cdot \xi^{\alpha+\beta-\epsilon} \eta^{\sigma+\tau-\epsilon},
\end{displaymath}
where $\binom{\alpha}{\epsilon} = \prod_{i=1}^n \binom{\alpha_i}{\epsilon_i}$ and $\epsilon! = \prod_{i=1}^n \epsilon_i!$.
Thus, we have
\begin{displaymath}
[a_1,a_2](\xi^{\alpha} \otimes \xi^{\beta} \otimes x) = \sum_{\sigma,\tau,\epsilon} \left( \binom{\beta}{\epsilon} \binom{\sigma}{\epsilon} - \binom{\alpha}{\epsilon} \binom{\tau}{\epsilon} \right) \epsilon! \cdot \xi^{[\sigma]} \otimes \xi^{[\tau]} \otimes \xi^{\alpha+\beta-\epsilon} \eta^{\sigma+\tau-\epsilon} x.
\end{displaymath}
We now compute $a'(\xi^{\alpha} \otimes \xi^{\beta} \otimes x)$. The map $a'$ is defined as the composition of seven maps. The effect of each is worked out in turn in the following derivation
\begingroup
\allowdisplaybreaks
\begin{align*}
\xi^{\alpha} \otimes \xi^{\beta} \otimes x
&\mapsto \sum_{\epsilon} \binom{\beta}{\epsilon} \xi^{\alpha} \otimes \xi^{\epsilon} \otimes \xi^{\beta-\epsilon} \otimes x \\
&\mapsto \sum_{\epsilon} \binom{\beta}{\epsilon} \xi^{\epsilon} \otimes \xi^{\alpha} \otimes \xi^{\beta-\epsilon} \otimes x \\
&\mapsto \sum_{\epsilon} \binom{\beta}{\epsilon} \xi^{\epsilon} \otimes \xi^{\alpha+\beta-\epsilon} \otimes x \\
&\mapsto \sum_{\epsilon, \rho} \binom{\beta}{\epsilon} \xi^{\epsilon} \otimes \xi^{[\rho]} \otimes \xi^{\alpha+\beta-\epsilon} \eta^{\rho} x \\
&\mapsto \sum_{\epsilon,\nu,\mu} \binom{\beta}{\epsilon} \xi^{\epsilon} \otimes \xi^{[\nu]} \otimes \xi^{[\mu]} \otimes \xi^{\alpha+\beta-\epsilon} \eta^{\nu+\mu} x \\
&\mapsto \sum_{\epsilon,\nu,\mu} \binom{\beta}{\epsilon} \epsilon! \cdot \xi^{[\epsilon]} \otimes \xi^{[\nu]} \otimes \xi^{[\mu]} \otimes \xi^{\alpha+\beta-\epsilon} \eta^{\nu+\mu} x \\
&\mapsto \sum_{\epsilon,\nu,\mu} \binom{\beta}{\epsilon} \binom{\epsilon+\nu}{\epsilon} \epsilon! \cdot \xi^{[\epsilon+\nu]} \otimes \xi^{[\mu]} \otimes \xi^{\alpha+\beta-\epsilon} \eta^{\nu+\mu} x \\
&= \sum_{\epsilon,\sigma,\tau} \binom{\beta}{\epsilon} \binom{\sigma}{\epsilon} \epsilon! \cdot \xi^{[\sigma]} \otimes \xi^{[\tau]} \otimes \xi^{\alpha+\beta-\epsilon} \eta^{\sigma+\tau-\epsilon} x.
\end{align*}
\endgroup
We thus see that $a'$ gives the first term in $[a_1,a_2]$. A similar computation shows that $a''$ gives the second, which completes the proof.
\end{proof}

\begin{definition}
Let $\cC$ be a tensor category and let $V$ be an object of $\cC$. We define the \textbf{curried Weyl Lie algebra $\ul{\fa}(V \oplus V^*)$} as follows. A representation of $\ul{\fa}(V \oplus V^*)$ is an object $M$ of $\cC$ equipped with maps
\begin{displaymath}
a_{n,m} \colon \Sym^n(V) \otimes M \to \Div^m(V) \otimes M
\end{displaymath}
for all $n,m \ge 0$, such that $[a_1,a_2]=a'-a''$, where $a'$ and $a''$ are defined as in the previous section.
\end{definition}

\begin{proposition} \label{prop:weyl-to-witt}
Let $V$ be an object of $\cC$ and let $(M,\alpha)$ be a representation of $\ul{\fa}(V \oplus V^*)$. Let $a$ be the composition
\begin{displaymath}
\xymatrix{
\Sym(V) \otimes M \ar[r]^\alpha & \Div(V) \otimes M \ar[r]^-{\pi \otimes \id} & V \otimes M}
\end{displaymath}
where the second map comes from the projection $\pi \colon \Div(V) \to V$. Then $a$ is a representation of the curried Witt algebra $\ul{W}(V)$. In particular, the composition
\begin{displaymath}
\xymatrix{
V \otimes M \ar[r] & \Sym(V) \otimes M \ar[r]^\alpha & \Div(V) \otimes M \ar[r] & V \otimes M}
\end{displaymath}
is a representation of $\ul{\fgl}(V)$.
\end{proposition}

\begin{proof}
  First, we have
  \[
    [a_1,a_2] = (\pi \otimes \pi \otimes \id) \circ [\alpha_1,\alpha_2].
  \]
  Second, $a'$ is equal to the following composition:
  \begin{align*}
  \Sym V \otimes \Sym V \otimes M &\xrightarrow{{\rm id} \otimes \Delta \otimes {\rm id}} \Sym V \otimes \Sym V \otimes \Sym V \otimes M\\
  &\xrightarrow{\tau \otimes {\rm id} \otimes {\rm id}}  \Sym V \otimes \Sym V \otimes \Sym V \otimes M\\
                                  &\xrightarrow{{\rm id} \otimes m \otimes {\rm id}} \Sym V \otimes \Sym V \otimes M\\
                                  &\xrightarrow{\id \otimes \alpha} \Sym V \otimes \Div V \otimes M\\
    &\xrightarrow{\pi \otimes \pi \otimes \id} V \otimes V\otimes M.
  \end{align*}
  The first four maps agree with the first four maps of the definition of $\alpha'$. It is straightforward to verify that the map
  \[
    \pi \otimes \pi \otimes \id \colon \Sym V \otimes \Div V \otimes M \to V \otimes V \otimes M
  \]
  agrees with the composition
  \begin{align*}
    \Sym V \otimes \Div V \otimes M &\xrightarrow{\id \otimes \Delta \otimes \id} \Sym V \otimes \Div V \otimes \Div V \otimes M\\
                                    &\xrightarrow{\avg \otimes \id \otimes \id \otimes \id}\Div V \otimes \Div V \otimes \Div V \otimes M\\
                                    &\xrightarrow{\id \otimes m \otimes \id} \Div V \otimes \Div V \otimes M\\
                                    &\xrightarrow{\pi \otimes \pi \otimes \id} V \otimes V \otimes M.
  \end{align*}
  In particular, we have $a_1 = (\pi \otimes \pi \otimes \id) \circ \alpha_1$ and by applying $\tau$, we conclude that $a_2 = (\pi \otimes \pi \otimes \id) \circ \alpha_2$. This means $[a_1,a_2]=a'-a''$ is a result of applying $(\pi \otimes \pi \otimes \id)$ to the identity $[\alpha_1,\alpha_2] = \alpha'-\alpha''$.
\end{proof}

\subsection{In species}

Let $M$ be an $\FB$-module and consider a map
\begin{displaymath}
a \colon \Sym(\bV) \otimes M \to \Sym(\bV) \otimes M.
\end{displaymath}
Let $\phi$ be the corresponding symmetric operation on $M$. Thus if $S$ is a finite set, $B$ is a subset of $S$, and $x$ is an element of $M(S \setminus B)$, then
\begin{displaymath}
a(t^B \otimes x) = \sum_{A \subseteq S} t^A \otimes \phi^S_{A,B}(x).
\end{displaymath}
We consider the following conditions on $\phi$. Let $A$, $B$, $C$, and $D$ be subsets of a finite set $S$, with $A \cap B = \emptyset$ and $C \cap D = \emptyset$.
\begin{itemize}
\item[(B1)] If $A \cap C=B \cap D=\emptyset$ then
\begin{displaymath}
\phi^{S \setminus C}_{D, A} \circ \phi^{S \setminus A}_{C,B}
= \phi^{S \setminus D}_{C,B} \circ \phi^{S \setminus B}_{D,A}.
\end{displaymath}
In other words, $\phi$ commutes with itself.
\item[(B2)] If $A \cap C=\emptyset$ and $B \cap D \ne \emptyset$ then
\begin{displaymath}
\phi^{S \setminus C}_{D, A} \circ \phi^{S \setminus A}_{C,B}
= \sum_{\substack{X \subseteq B \cap D\\ X \ne \emptyset}} \phi^{S \setminus X}_{(D \setminus X) \cup C,A \cup (B \setminus X)}.
\end{displaymath}
We remark that there is also a version of this condition if $A \cap C \ne \emptyset$ and $B \cap D = \emptyset$; however, since the above condition holds for all choices of $A,B$ and $C,D$ and they play symmetric roles, we omit listing it separately as it is actually redundant.
\item[(B3)] If $A \cap C \ne \emptyset$ and $B \cap D \ne \emptyset$ then
\begin{displaymath}
\sum_{X \subseteq B \cap D} \phi^{S \setminus X}_{(D \setminus X) \cup C,A \cup (B \setminus X)}
= \sum_{X \subseteq A \cap C} \phi^{S \setminus X}_{(C \setminus X) \cup D, B \cup (A \setminus X)}.
\end{displaymath}
\end{itemize}
We then have the result:

\begin{proposition}
The map $a$ defines a representation of $\ul{\fa}(\bV \oplus \bV^*)$ if and only if the operation $\phi$ satisfies (B1), (B2), and (B3).
\end{proposition}

\begin{proof}
Let $A$ and $B$ be disjoint subsets of $S$ and let $x \in M(S \setminus (A \cup B))$. Then
\begin{align*}
a_1(a_2(t^A \otimes t^B \otimes x)) &= \sum_{\substack{C \amalg D \subseteq S\\ A \cap C = \emptyset}} t^{[D]} \otimes t^{[C]} \otimes \phi^{S \setminus C}_{D, A}(\phi^{S \setminus A}_{C,B}(x)),\\
a_2(a_1(t^A \otimes t^B \otimes x)) &= \sum_{\substack{C \amalg D \subseteq S\\ B \cap D = \emptyset}} t^{[D]} \otimes t^{[C]} \otimes \phi^{S \setminus D}_{C,B}(\phi^{S \setminus B}_{D,A}(x)),\\
a'(t^A \otimes t^B \otimes x) &= \sum_{C \amalg D \subseteq S} \sum_{X \subseteq B \cap D} t^{[D]} \otimes t^{[C]} \otimes \phi^{S \setminus X}_{(D \setminus X) \cup C,A \cup (B \setminus X)}(x),\\
a''(t^A \otimes t^B \otimes x) &= \sum_{C \amalg D \subseteq S} \sum_{X \subseteq A \cap C} t^{[D]} \otimes t^{[C]} \otimes \phi^{S \setminus X}_{(C \setminus X) \cup D, B \cup (A \setminus X)}(x).
\end{align*}
Equating coefficients, one sees that $[a_1,a_2]=a'-a''$ if and only if (B1), (B2), and (B3) hold.
\end{proof}

Recall that a symmetric operation $\phi$ corresponds to a sequence $(\phi[n])_{n \ge 0}$ of simple symmetric operations. The correspondence is given by $\phi[n]^S_{A,B}=\phi^{S \amalg [n]}_{A \amalg [n], B \amalg [n]}$. We now wish to translate the conditions (B1), (B2), and (B3) to the $\phi[n]$. We begin with the following observation:

\begin{proposition}
Condition {\rm (B3)} is equivalent to the following condition:
\begin{itemize}
\item[\rm (B3$'$)] We have $\phi[n]=(-1)^{n+1} \phi[1]$ for all $n \ge 1$.
\end{itemize}
\end{proposition}

\begin{proof}
Suppose (B3) holds. Let $P$ and $Q$ be disjoint subsets of a set $S$. Let $r \ge 0$ and put
\begin{displaymath}
\tilde{S}=S \amalg \{i_1, \ldots, i_r, j_1, j_2, k \}
\end{displaymath}
where the $i$'s, $j$'s, and $k$ are distinct from each other and all elements of $S$. Put
\begin{displaymath}
A = Q \cup \{i_1, \ldots, i_r, k \}, \quad
B = \{j_1, j_2\}, \quad
C = P \cup \{k\}, \quad
D = \{i_1,\ldots,i_r,j_1,j_2\}.
\end{displaymath}
We have
\begin{align*}
\sum_{X \subseteq B \cap D} \phi^{\tilde{S} \setminus X}_{(D \setminus X) \cup C, A \cup (B \setminus X)} &= \phi[r+3]^S_{P,Q}+2\phi[r+2]^S_{P,Q}+\phi[r+1]^S_{P,Q} \\
\sum_{X \subseteq A \cap C} \phi^{\tilde{S} \setminus X}_{D \cup (C \setminus X), (A \setminus X) \cup B} &= \phi[r+3]^S_{P,Q}+\phi[r+2]^S_{P,Q}.
\end{align*}
By (B3), the above two expressions are equal. We thus find
\begin{displaymath}
\phi[r+2]=-\phi[r+1].
\end{displaymath}
As this holds for all $r \ge 0$, we find $\phi[n]=(-1)^{n+1} \phi[1]$ for $n \ge 1$, and so (B3$'$) holds.

Now suppose (B3$'$) holds. This implies
\begin{displaymath}
\phi^{S \amalg Y}_{P \amalg Y, Q \amalg Y}=(-1)^{\# Y} \phi^S_{P,Q}
\end{displaymath}
provided that $P$ and $Q$ are not disjoint. Let $A$, $B$, $C$, and $D$ be as in (B3). Put $m=\# (B \cap D)$, and suppose $X$ is a subset of $B \cap D$ of size $k$. Then applying the above equation with $Y=(B \cap D) \setminus X$, we find
\begin{displaymath}
\phi^{S \setminus X}_{(D \setminus X) \cup C, A \cup (B \setminus X)}=
(-1)^{m-k} \phi^{S \setminus (B \cup D)}_{(D \setminus B) \cup C, A \cup (B \setminus D)} =0.
\end{displaymath}
Note that $A \cap C \ne \emptyset$ since we are in the setting of (B3). It follows that
\begin{displaymath}
\sum_{X \subseteq B \cap D} \phi^{S \setminus X}_{(D \setminus X) \cup C, A \cup (B \setminus X)}
=\sum_{k=0}^m \binom{m}{k} (-1)^{m-k} \phi^{S \setminus (B \cup D)}_{(D \setminus B) \cup C, A \cup (B \setminus D)} =0.
\end{displaymath}
Similarly, the other sum in (B3) vanishes, and so (B3) holds.
\end{proof}

The above proposition shows that we just need to understand the operations $\phi[0]$ and $\phi[1]$. To this end, we introduce some notation. Let $\sB_M$ denote the set of symmetric operations $\phi$ on $M$ satisfying (B1), (B2), and (B3), and let $\sC_M$ denote the set of pairs $(\alpha, \omega)$ of simple symmetric operations on $M$ satisfying the following conditions (C1) and (C2). In what follows, $A$, $B$, $C$, and $D$ are subsets of a finite set $S$.
\begin{itemize}
\item[(C1)] The operations $\alpha$ and $\omega$ commute with themselves and each other. Precisely, assuming that $A$, $B$, $C$ and $D$ are pairwise disjoint, we have
\begin{displaymath}
\alpha^{S\setminus C}_{D,A} \circ \omega^{S \setminus A}_{C,B} = \omega^{S \setminus D}_{C,B} \circ \alpha^{S \setminus B}_{D,A},
\end{displaymath}
and similarly with $\omega$ replaced by $\alpha$, or $\alpha$ replaced by $\omega$.
\item[(C2)] Suppose $B \cap D \ne \emptyset$, but all other pairs disjoint. Put
\begin{align*}
\alpha_1 &= \alpha^{S \setminus C}_{D,A} &
\alpha_2 &= \alpha^{S \setminus A}_{C,B} &
\alpha_3 &= \alpha^{S \setminus (B \cap D)}_{C \cup (D \setminus B),A \cup (B \setminus D)} \\
\omega_1 &= \omega^{S \setminus C}_{D,A} &
\omega_2 &= \omega^{S \setminus A}_{C,B} &
\omega_3 &= \omega^{S \setminus (B \cap D)}_{C \cup (D \setminus B),A \cup (B \setminus D)}
\end{align*}
and let $m=\# (B \cap D)$. Then
\begin{displaymath}
\alpha_1 \alpha_2=\alpha_3, \qquad \alpha_1 \omega_1=\alpha_2 \omega_1=0, \qquad \omega_1 \omega_2=(-1)^{m+1} \omega_3.
\end{displaymath}
\end{itemize}
We now have the following:

\begin{proposition} \label{prop:B-C-bijec}
We have a bijection
\begin{displaymath}
\Theta \colon \sB_M \to \sC_M, \qquad \phi \mapsto (\phi[0]+\phi[1],-\phi[1]).
\end{displaymath}
The inverse to $\Theta$ can be described as follows: $\phi=\Theta^{-1}(\alpha,\omega)$ is the unique symmetric operation on $M$ satisfying $\phi[0]=\alpha+\omega$ and $\phi[n]=(-1)^n \omega$ for $n \ge 1$.
\end{proposition}

We first prove a lemma.

\begin{lemma}
Suppose that $\phi$ satisfies {\rm (B3)}. Then {\rm (B2)} is equivalent to the following condition:
\begin{itemize}
\item[\rm (B$2'$)] Let $A$, $B$, $C$, and $D$ be subsets of a finite set $S$ such that $B \cap D \ne \emptyset$, but all other pairs are disjoint. Put $m=\# (B \cap D)$. Then we have
\begin{align*}
\phi[0]^{S \setminus C}_{D,A} \circ \phi[0]^{S \setminus A}_{C,B}  &= \phi[0]^{S \setminus (B \cap D)}_{(D \setminus B) \cup C, A \cup (B \setminus D)} + (1+(-1)^m) \phi[1]^{S \setminus (B \cap D)}_{(D \setminus B) \cup C, A \cup (B \setminus D)}, \\
\phi[p]^{S \setminus C}_{D,A} \circ \phi[q]^{S \setminus A}_{C,B} &= (-1)^{p+q+m} \phi[1]^{S \setminus (B \cap D)}_{(D \setminus B) \cup C, A \cup (B \setminus D)},
\end{align*}
where in the second equation $p$ and $q$ are non-negative and not both zero.
\end{itemize}
\end{lemma}

\begin{proof}
Suppose $\phi$ satisfies (B2). Let $A$, $B$, $C$, and $D$ be as above. Let $P$ and $Q$ be sets disjoint from each other and from $S$ of cardinalities $p$ and $q$. Put
\begin{displaymath}
S'=S \cup P \cup Q, \quad
A'=A \cup P, \quad
B'=B \cup Q, \quad
C'=C \cup Q, \quad
D'=D \cup P.
\end{displaymath}
Let $m=\# (B \cap D)$. Applying (B2) to the prime sets, we find
\begin{displaymath}
\phi[p]^{S \setminus C}_{D,A} \circ \phi[q]^{S \setminus A}_{C,B} = \sum_{k=1}^m \binom{m}{k} \phi[p+q+m-k]^{S \setminus (B \cap D)}_{(D \setminus B) \cup C, A \cup (B \setminus D)}.
\end{displaymath}
Now, if $p+q>0$ then by (B3$'$) $\phi[p+q+m-k]=(-1)^{p+q+m-k+1} \phi[1]$, and we obtain the second equation in (B2$'$). If $p+q=0$ then (B3$'$) only gives this identity for $0 \le k<m$, and so the final term in the above sum must be handled differently; this gives the first equation in (B2$'$). Thus (B2$'$) holds. The same reasoning yields the reverse implication.
\end{proof}

\begin{proof}[Proof of Proposition~\ref{prop:B-C-bijec}]
Let $\phi \in \sB_M$ be given, and put $\alpha=\phi[0]+\phi[1]$ and $\omega=-\phi[1]$. Condition (C1) follows immediately from (B1). We now examine condition (C2); use the notation from there. Translating (B2$'$) to this notation gives
\begin{align*}
(\alpha_1+\omega_1)(\alpha_2+\omega_2) &= (\alpha_3+\omega_3)-(1+(-1)^m) \omega_3, \\
(\alpha_1+\omega_1) \omega_2 &= (-1)^{m+1} \omega_3, \\
\omega_1 (\alpha_2+\omega_2) &= (-1)^{m+1} \omega_3, \\
\omega_1 \omega_2 &= (-1)^{m+1} \omega_3.
\end{align*}
Here the first equation is the first equation from (B2$'$), while the final three equations come from taking $(p,q)$ to be $(0,1)$, $(1,0)$, and $(1,1)$ in (B2$'$). One easily sees that the above equations yield those from (C2). Thus $(\alpha, \omega)$ belongs to $\sC_M$. The above reasoning is reversible.
\end{proof}

\subsection{The partition category} \label{ss:part}

Let $\fP=\fP(\delta)$ be the partition category with parameter $\delta$ (see \S \ref{ss:res-part}). Suppose that $M$ is a $\fP$-module. Then $M$ restricts to an $\FB$-module. Given a finite set $S$ and disjoint subsets $A$ and $B$ such that at least one is non-empty, let $\eta^S_{A,B} \colon S \setminus B \to S \setminus A$ be the morphism in $\fP$ in which $A \cup B$ forms a single block, and the remaining diagram is the identity. By convention, $\eta^S_{\emptyset,\emptyset}=\delta \cdot {\rm id}_S$. Let
\begin{displaymath}
\alpha^S_{A,B} \colon M(S \setminus B) \to M(S \setminus A)
\end{displaymath}
be the induced map. One easily sees that $\sigma$ is a simple symmetric operation on $M$. Using the rule for composition in $\fP$, we find that the following conditions hold:
\begin{itemize}
\item[(D1)] $\alpha$ commutes with itself in the sense of (C1).
\item[(D2)] We have $\alpha_1 \alpha_2=\alpha_3$ in the notation of (C2).
\item[(D3)] $\alpha^S_{\emptyset,\emptyset}=\delta$ for all finite sets $S$.
\item[(D4)] For distinct elements $i,j \in S$, we have $\alpha^S_{\{i\},\{j\}}=(\iota^S_{i,j})_*$.
\end{itemize}

\begin{proposition} \label{prop:part-op}
Let $M$ be an $\FB$-module equipped with a simple symmetric operation $\sigma$ satisfying the above conditions. Then $M$ carries a unique $\fP$-structure inducing $\sigma$.
\end{proposition}

\begin{proof}
First, we use $\alpha$ to construct $\fU$ and $\fD$-structures on $M$. The operation $\alpha$ gives, by restriction, a simple $(1,*)$-operation and a simple $(0,*)$-operation. We can use these to define an action of the upwards restricted partition category $\fU^r$ on $M$ as in the proof of Proposition~\ref{prop:res-part-op}, and we note that $\fU^r=\fU$. Similarly, we get, by restriction, a simple $(*,1)$-operation and a simple $(*,0)$-operation which gives a $\fD$-structure on $M$. (D4) tells us that the restriction of the $\fD$ and $\fU$ structures to $\FB$ agree with the usual $\FB$-action, so in particular they agree with each other.

Let $\cU$ be the class of morphisms in $\fU$ isomorphic to $\eta^S_{A,B}$ for some $S$, $A$, and $B$ with $|A| \le 1$, and define $\cD$ similarly using $\eta^S_{A,B}$ with $|B| \le 1$. One easily sees that $\cU$ generates $\fU$ and $\cD$ generates $\fD$. Thus, by Proposition~\ref{prop:tri-comp}, it suffices to show that $(\phi, \psi)$ is compatible for $\phi \in \cU$ and $\psi \in \cD$ with $\psi \circ \phi$ defined.

Let $A,B,C,D$ be subsets of $S$ such that $|C|\le 1$ and $|A|\le 1$ and all pairs of subsets are disjoint, except possibly $B$ and $D$. We set $\phi = \eta^{S\setminus A}_{C,B}$ and $\psi = \eta^{S \setminus C}_{D,A}$. If $B \cap D = \emptyset$, then compatibility follows from (D1), i.e., the condition that $\alpha$ is self-commuting. Otherwise, suppose that $B \cap D \ne \emptyset$. Then compatibility follows from (D2) if at least one of $C \cup (D \setminus B)$ and $A \cup (B\setminus D)$ is non-empty. If both are empty, then we also need to use (D3).
\end{proof}

Suppose that $\cC$ and $\cD$ are two $\bk$-linear categories whose objects are finite sets and which contain all bijections. We define a new $\bk$-linear category $\cC \star \cD$ whose objects are finite sets as follows. The Hom spaces are defined by
\begin{displaymath}
\Hom_{\cC \star \cD}(S,T)=\bigoplus_{\substack{S=S_1 \sqcup S_2 \\ T= T_1 \sqcup T_2}} \Hom_{\cC}(S_1, T_1) \otimes \Hom_{\cD}(S_2, T_2).
\end{displaymath}
Suppose that $S \to T$ and $T \to U$ are morphisms in $\cC \star \cD$ corresponding to decompositions $S=S_1 \sqcup S_2$ and $T=T_1 \sqcup T_2$, and $T=T_1' \sqcup T_2'$ and $U=U_1 \sqcup U_2$. If $T_1=T_1'$ and $T_2=T_2'$ the composition is defined in the obvious manner, using the composition laws in $\cC$ and $\cD$; otherwise, the composition is defined to be 0. We have a functor $\FB \to \cC \star \cD$ that is the identity on objects and takes a bijection $\phi \colon S \to T$ to
\begin{displaymath}
\sum_{S=S_1 \sqcup S_2} (\phi \colon S_1 \to \phi(S_1)) \otimes (\phi \colon S_2 \to \phi(S_2)).
\end{displaymath}
In particular, the identity morphism of $S$ in $\cC \star \cD$ is $\sum_{S=A \sqcup B} \id_{\cC,A} \otimes \id_{\cD,B}$. There is no natural functor $\cC \to \cC \star \cD$ (the obvious attempt does not preserve identity morphisms), but there is a natural functor $\cC \star \cD \to \cC$ which kills all morphisms in $\cD$. Similarly, there is a natural functor $\cC \star \cD \to \cD$ which kills all morphisms in $\cC$. We will apply this construction with $\cC=\fP(\delta)$ and $\cD=\fP(\epsilon)$ below.

\subsection{The comparison theorem}

Let $M$ be an $\ul{\fa}(\bV \oplus \bV^*)$-module. We say that $M$ is \defi{$\delta$-standard} if its restriction to $\ul{\fgl}(V)$ (see Proposition~\ref{prop:weyl-to-witt}) is $\delta$-standard. We say that $M$ has \defi{central character} $\chi \in \bk$ if the composition
\begin{displaymath}
M \to \Sym(V) \otimes M \to \Div(V) \otimes M \to M
\end{displaymath}
is $\chi$ times the identity, where the first map is the natural isomorphism $M \to \Sym^0(V) \otimes M$ followed by the inclusion into $\Sym(V) \otimes M$ while the second map is the projection $\Div(V) \otimes M \to \Div^0(V) \otimes M$ followed by the natural isomorphism with $M$.

We let $\Rep^{\chi}_{\delta}(\ul{\fa}(\bV \oplus \bV^*))$ be the full subcategory of $\ul{\fa}(\bV \oplus \bV^*)$ spanned by $\delta$-standard modules with central character $\chi$.

\begin{theorem}
We have an equivalence of categories
\begin{displaymath}
\Rep_{\epsilon}^{\delta-\epsilon}(\ul{\fa}(\bV \oplus \bV^*)) = \Rep(\fP(\delta) \star \fP(\epsilon)).
\end{displaymath}
\end{theorem}

\begin{proof}
  Let $M$ be a representation of $\fP(\delta) \star \fP(\epsilon)$. Then $M$ restricts to an $\FB$-module via the functor $\FB \to \fP(\delta) \star \fP(\epsilon)$ defined above. Let $S$ be a finite set with disjoint subsets $A$ and $B$. If $A \cup B \ne \emptyset$, we define
  \[
    \alpha^S_{A,B} \colon M(S \setminus B) \to M(S \setminus A)
  \]
  using the morphism in $\fP(\delta) \star \fP(\epsilon)$ that is the identity on $S \setminus (A \cup B)$ and the morphism $A \to B$ in $\fP(\delta)$ given by a single block. We define $\alpha^S_{\emptyset, \emptyset}$ to be $\delta$ times the identity on $M(S)$. Similarly, if $A \cup B \ne \emptyset$, we define
  \[
    \omega^S_{A,B} \colon M(S \setminus B) \to M(S\setminus A)
  \]
  to be $(-1)^{|A|+1}$ times the morphism which is the identity on $S \setminus (A\cup B)$ and the morphism $A \to B$ in $\fP(\epsilon)$, and we define $\alpha^S_{\emptyset, \emptyset}$ to be $-\epsilon$ times the identity on $M(S)$. It follows from the above discussion that $\alpha$ and $\omega$ satisfy conditions (C1) and (C2) and thus define a representation of $\ul{\fa}(V \oplus V^*)$ on $M$.

Let $\phi \in \sB_M$ be the operation associated to $(\alpha,\omega)$. Thus $\phi[0]=\alpha+\omega$ and $\phi[1]=-\omega$. We have
\begin{displaymath}
\phi^S_{\emptyset,\emptyset}=\alpha^S_{\emptyset,\emptyset}+\omega^S_{\emptyset,\emptyset}=\delta - \epsilon
\end{displaymath}
and so we see that $M$ has central character $\delta - \epsilon$. Let $i,j \in S$. For $i=j$, we have
\begin{displaymath}
\phi^S_{i,i} = \phi[1]^{S \setminus x}_{\emptyset,\emptyset} = \epsilon.
\end{displaymath}
For $i \ne j$, we have
\begin{displaymath}
\phi^S_{i,j} = \alpha^S_{i,j}+\omega^S_{i,j} = (\iota^S_{i,j})_*,
\end{displaymath}
where $\iota^S_{i,j} \colon S \setminus j \to S \setminus i$ is the natural bijection; the second equality follows from the definition of the $\FB$-structure on $M$. We thus see that the $\ul{\fgl}(V)$-module structure on $M$ is given by
\begin{displaymath}
t^i \otimes x \mapsto \epsilon x + \sum_{j \ne i} t^j \otimes (\iota_{i,j}^S)_*(x)
\end{displaymath}
where $x \in M(S \setminus i)$. Hence the action is $\epsilon$-standard.
\end{proof}

Let $\Rep^{\delta}(\ul{\fa}(\bV \oplus \bV^*))'$ be the full subcategory of $\Rep(\ul{\fa}(\bV \oplus \bV^*))$ spanned by representations that have central character $\delta$, are 0-standard, and for which the $\omega$ operations vanish.

\begin{corollary}
We have a natural isomorphism of categories
\begin{displaymath}
\Rep^{\delta}(\ul{\fa}(\bV \oplus \bV^*))' \cong \Rep(\fP(\delta)).
\end{displaymath}
\end{corollary}

\section{Abstract curried algebras} \label{s:abs}

\subsection{The definition}

We have defined the notion of representation for several curried algebras. However, we have not given a general definition of curried algebra. We now briefly (and informally) give such a definition. It would be interesting to explore this idea in more detail.

Let $\cC$ be a symmetric monoidal $\bk$-linear category. We assume that the monoidal structure $\otimes$ is $\bk$-bilinear. Let $M$ be an object of $\cC$. Given two other objects $V$ and $W$ of $\cC$, a \defi{$(V,W)$-operation} on $M$ is a map
\begin{displaymath}
a \colon V \otimes M \to W \otimes M.
\end{displaymath}
Intuitively, giving a curried algebra should amount to giving some $(V,W)$-operations (for various $V$ and $W$) satisfying some relations. ``Relations'' will mean that certain operations built out of these operations vanish. Given a $(V,W)$-operation $a$, there are four ways to build new operations:
\begin{enumerate}
\item Tensor with an arbitrary object $X$ to obtain an $(X \otimes V, X \otimes W)$-operation:
\begin{displaymath}
\xymatrix@C=4em{
X \otimes V \otimes M \ar[r]^-{\id \otimes a} & X \otimes W \otimes M. }
\end{displaymath}
\item Pre-compose with a morphism $f \colon V' \to V$ and post-compose with a morphism $g \colon W \to W'$ to obtain a $(V',W')$-operation:
\begin{displaymath}
\xymatrix@C=4em{
V' \otimes M \ar[r]^{f \otimes \id} & V \otimes M \ar[r]^-a & W \otimes M \ar[r]^{g \otimes \id} & W' \otimes M. }
\end{displaymath}
\item Given a $(U,V)$-operation $b$, compose to obtain a $(U,W)$-operation:
\begin{displaymath}
\xymatrix@C=4em{
U \otimes M \ar[r]^b & V \otimes M \ar[r]^a & W \otimes M.}
\end{displaymath}
\item Given a finite collection of $(V,W)$-operations $\{a_i\}$, any $\bk$-linear combination $\sum_i \lambda_i a_i$ is also a $(V,W)$-operation.
\end{enumerate}
This suggests the following definition:

\begin{definition}
A \defi{curried algebra} $A$ in $\cC$ consists of the following primary data:
\begin{itemize}
\item For each pair of objects $(V,W)$ in $\cC$, a $\bk$-vector space $A(V,W)$. This is called the space of $(V,W)$-operations in $A$. For $V=W$, there is a distinguished ``identity operation'' in $A(V,V)$.
\end{itemize}
Additionally, we require the following operations on $A$:
\begin{enumerate}
\item For objects $V$, $W$, and $X$, a $\bk$-linear map $A(V, W) \to A(X \otimes V, X \otimes W)$.
\item For morphisms $V' \to V$ and $W \to W'$, a $\bk$-linear map $A(V,W) \to A(V',W')$.
\item For objects $U$, $V$, and $W$, A $\bk$-linear map $A(U,V) \otimes A(V,W) \to A(U,W)$.
\end{enumerate}
A number of conditions should hold (that we do not specify).
\end{definition}

\begin{remark} \label{rmk:curcat}
Let $A$ be a curried algebra in $\cC$. One can then define a $\bk$-linear category $\cD$ with the same objects as $\cC$ and with $\Hom_{\cD}(V,W)=A(V,W)$. Composition is given by the operation in (c). The operations in (a) and (b) define an action of the monoidal category $\cC$ on $\cD$, i.e., a functor $\cC \times \cD \to \cD$ satisfying certain conditions. In fact, giving $A$ is equivalent to giving $\cD$ (together with this action), and so one can view curried algebras as certain kinds of categories.
\end{remark}

\subsection{Constructions}

Assuming $\cC$ satisfies some mild conditions, there are two general constructions of curried algebras:
\begin{itemize}
\item Given a collection $\{(V_i,W_i)\}_{i \in I}$ of pairs of objects in $\cC$, there is a free curried algebra containing a distinguished $(V_i,W_i)$-operation for each $i \in I$.
\item Given a curried algebra $A$ and a collection of elements $\{x_i \in A(V_i,W_i)\}_{i \in I}$, there is a quotient curried algebra $B$ in which each $x_i$ maps to~0 (and which is universal subject to this).
\end{itemize}
These two constructions allow one to build curried algebras by generators and relations. For instance, one can build the curried algebra $\ul{\fgl}(V)$ by taking the free curried algebra on a single $(V,V)$-operation $a$ and quotienting by the $(V \otimes V, V \otimes V)$-operation $[a_1,a_2]-\tau(a_1-a_2)$.

There is one additional important construction of curried algebras. Let $M$ be an object of $\cC$. We define the \defi{endomorphism curried algebra} of $M$, denoted $E_M$, by
\begin{displaymath}
E_M(V,W) = \Hom_{\cC}(V \otimes M, W \otimes M).
\end{displaymath}
If $A$ is an arbitrary curried algebra, then a \defi{representation} of $A$ on $M$ is a homomorphism of curried algebras $A \to M$. The representations of various curried algebras that we have discussed above all fit into this general framework.

\subsection{Diagram categories}

We have seen several examples of diagram categories in this paper, such as the Brauer category and the partition category. We now propose a precise definition of ``diagram category.''

To motivate the definition, suppose that $\fG$ is one of the familiar $\bk$-linear diagram categories, such as the Brauer category. In particular, the objects of $\fG$ are finite sets. If $T$ is a finite set, then there is a functor $\fG \to \fG$ given on objects by $S \mapsto S \amalg T$. On morphisms, this functor corresponds to adding vertices indexed by $T$ to the source and target, and connecting these vertices by lines. In other words, we see that $\fG$ admits an action by the monoidal category $\FB$. We now give our definition:

\begin{definition}
A \defi{diagram category} is a $\bk$-linear category whose objects are finite sets equipped with a $\bk$-linear action by the monoidal category $\FB$ that lifts disjoint union.
\end{definition}

Note that a diagram category in the above sense is just a curried algebra in $\FB$, from the point of view of Remark~\ref{rmk:curcat} (with the caveat that $\FB$ is not a $\bk$-linear category). The point of this paper can now be rephrased as follows: for many diagrams categories $\fG$, we associated a curried algebra $A$ in $\Mod_{\FB}$ such that $\fG$-modules and $A$-modules coincide.

\addtocontents{toc}{

\end{document}
\vskip 6pt}
\begin{thebibliography}{BCNR}

\bibitem[BS]{BrundanStroppel} Jonathan Brundan, Catharina Stroppel. Semi-infinite highest weight categories. {\it Mem. Amer. Math. Soc.}, to appear. \arxiv{1808.08022v4}

\bibitem[CEF]{fimodule} Thomas Church, Jordan S. Ellenberg, Benson Farb. FI-modules and stability for representations of symmetric groups. {\it Duke Math. J.} {\bf 164} (2015), no.~9, 1833--1910. \DOI{10.1215/00127094-3120274} \arxiv{1204.4533v4}

\bibitem[Ju]{jun} Hyung Kyu Jun. Representation theory of curried algebras and non-primality of certain symmetric ideals. (in preparation)
  
\bibitem[JS]{joyal-street} Andr\'e Joyal, Ross Street. The category of representations of the general linear groups over a finite field. \textit{J. Algebra} {\bf 176} (1995), no.~3, 908--946. \DOI{10.1006/jabr.1995.1278}

\bibitem[SS1]{symc1} Steven~V Sam, Andrew Snowden. GL-equivariant modules over polynomial rings in infinitely many variables. \textit{Trans. Amer. Math. Soc.} \textbf{368} (2016), 1097--1158. \DOI{10.1090/tran/6355} \arxiv{1206.2233v3}

\bibitem[SS2]{infrank} Steven V Sam, Andrew Snowden. Stability patterns in representation theory. \textit{Forum Math. Sigma} \textbf{3} (2015), e11, 108 pp. \DOI{10.1017/fms.2015.10} \arxiv{1302.5859v2}

\bibitem[SS3]{grobner} Steven V Sam, Andrew Snowden. Gr\"obner methods for representations of combinatorial categories. \textit{J.\ Amer.\ Math.\ Soc.} \textbf{30} (2017), 159--203. \DOI{10.1090/jams/859} \arxiv{1409.1670v3}

\bibitem[SS4]{brauercat1} Steven V Sam, Andrew Snowden. The representation theory of Brauer categories I: triangular categories. \arxiv{2006.04328v2}
  
\bibitem[SS5]{brauercat3} Steven V Sam, Andrew Snowden. The representation theory of Brauer categories III: category $\cO$. (in preparation)

\bibitem[SS6]{curried} Steven V Sam, Andrew Snowden. Examples of curried algebras. (in preparation)

\bibitem[SST]{witt} Steven V Sam, Andrew Snowden, Philip Tosteson. Polynomial representations of the Witt algebra. (in preparation)

\end{thebibliography}
